\documentclass[11pt, oneside]{amsart}   	
\usepackage{geometry}                		
\usepackage{amsmath,amsthm,amscd,amsfonts, amssymb,mathrsfs}

\usepackage{graphicx}				

\font\mbn=msbm10 scaled \magstep1
\font\mbs=msbm7 scaled \magstep1
\font\mbss=msbm5 scaled \magstep1
\newfam\mbff 
\textfont\mbff=\mbn
\scriptfont\mbff=\mbs
\scriptscriptfont\mbff=\mbss

\newtheorem{Th}{Theorem}[section]
\newtheorem{Lm}[Th]{Lemma}
\newtheorem{C}[Th]{Corollary}

\newtheorem{Proposition}[Th]{Proposition}
\newtheorem{R}[Th]{Remark}

\newtheorem{E}[Th]{Example}
\newtheorem*{Teo}{Theorem  A}

\title[Bernstein Type Inequalities for Restrictions of Polynomials]{Bernstein Type Inequalities for Restrictions of Polynomials to Complex Submanifolds of ${\mathbf {\mathbb C^N}}$ }
\author{Alexander Brudnyi}
\address{Department of Mathematics and Statistics\newline
\hspace*{1em} University of Calgary\newline
\hspace*{1em} Calgary, Alberta\newline
\hspace*{1em} T2N 1N4 Canada}
\email{abrudnyi@ucalgary.ca}

\thanks{Research supported in part by NSERC}

\date{}							

\begin{document}

\begin{abstract}
The paper studies Bernstein type inequalities for restrictions of holomorphic polynomials to graphs $\Gamma_f\subset\mathbb C^{n+m}$ of holomorphic maps $f:\mathbb C^n\rightarrow\mathbb C^m$.  We establish general properties of exponents in such inequalities and describe some classes of graphs admitting Bernstein type inequalities of optimal exponents and of exponents of polynomial growth.  
\end{abstract}

\maketitle

\section{Formulation of Main Results}
\subsection{} In recent years there was a considerable interest in Bernstein, Markov and Remez type inequalities for restrictions of holomorphic polynomials to certain submanifolds of $\mathbb C^N$  in connection with various problems of analysis and geometry, see, e.g., \cite{B, BBL, BBLT, BLMT, BP, CP1, CP2, CP3, FN1, FN2, NSV, P, RY, S} and references therein. 
Specifically, the graph $\Gamma_f\subset\mathbb C^{n+m}$ of a holomorphic map $f:\mathbb C^n\rightarrow\mathbb C^m$ is said to admit the {\em Bernstein type inequality of exponent} $\mu:\mathbb Z_+\rightarrow\mathbb R_+$ if for each $r>0$ there exists a nonnegative constant $C(r)$ such that for all holomorphic polynomials $p$ on $\mathbb C^{n+m}$
\begin{equation}\label{e1}
\max_{\|z\|\le er} |p(z,f(z))|\le C(r)^{\mu({\rm deg}\,p)} \,\max_{\|z\|\le r} |p(z,f(z))|.
\end{equation}
(Here $\|\cdot\|$ is the Euclidean norm on $\mathbb C^{n}$.)\smallskip

The value $\mu({\rm deg}\, p)\ln C(r)$ can be regarded as the degree of the function $p_f:=p(\cdot, f(\cdot))$. In particular, inequality \eqref{e1} implies the corresponding Markov and Remez type inequalities for functions $p_f$ in the Euclidean balls $\{z\in\mathbb C^n\, :\, \|z\|\le r\}$ with degrees of polynomials in the standard setting (see \cite{Ma, Be, R, BG}) replaced by $c\mu({\rm deg}\,p)\ln C(r)$ for an absolute constant $c>0$, see, e.g., \cite[Sect.~2]{B}, \cite{BLMT} for details.
In addition, if $n=1$, inequality \eqref{e1} implies the Jensen type inequality asserting that the number of zeros (counted with their multiplicities)
of  the function $p_f$ in the disc $\{z\in\mathbb C\, :\, |z|\le r\}$ is bounded from above by $\frac 52\mu({\rm deg}\,p)\ln C(r)$, see, e.g., \cite{VP}.

It is known that $\Gamma_f$ admits the Bernstein type inequality of exponent  $\mu_{{\rm id}}(k):=k$, $k\in\mathbb Z_+$, if and only if it is a complex algebraic manifold, see \cite{S}. On the other hand, it is easy to give examples of graphs $\Gamma_f$ for which the exponent $\mu$ in \eqref{e1} must be of an arbitrarily prescribed growth, see, e.g., \cite[p.\,140]{BBL}.  

In this paper we begin the systematic study of general properties of exponents in Bernstein type inequalities and of some classes of graphs $\Gamma_f$ admitting such inequalities of exponents of polynomial growth. Some of our proofs rely heavily upon the results of \cite{B}.
\subsection{}
In this part we describe some properties of exponents in Bernstein type inequalities.\smallskip

Recall that a subset $K\subset\mathbb C^n$ is called {\em pluripolar} if there exists a nonidentical $-\infty$ plurisubharmonic function $u$ on $\mathbb C^n$ such that 
$u|_{K}=-\infty$. (For basic results of the theory of plurisubharmonic functions see, e.g., \cite{K}.)
\begin{Th}\label{theo1}
\begin{itemize}
\item[(a)]
For each holomorphic map $f:\mathbb C^n\rightarrow\mathbb C^m$ its
graph $\Gamma_f\subset\mathbb C^{n+m}$ admits the Bernstein type inequality of certain exponent.\smallskip
\item[(b)]
$\Gamma_f\subset\mathbb C^{n+m}$ admits the Bernstein type inequality of exponent $\mu$ if and only if for each compact nonpluripolar subset $K\subset\mathbb C^n$ there exists a constant $C(K;r)$, $r>0$, such that for all $p\in\mathcal P(\mathbb C^{n+m})$, the space of holomorphic polynomials on $\mathbb C^{n+m}$,
\begin{equation}\label{e2}
\max_{\|z\|\le r} |p(z,f(z))|\le C(K;r)^{\mu({\rm deg}\,p)} \max_{z\in K} |p(z,f(z))|.\smallskip
\end{equation}
\item[(c)]
If $\Gamma_f\subset\mathbb C^{n+m}$ is nonalgebraic and admits the Bernstein type inequality of exponent $\mu$, then
\[
\varliminf_{k\rightarrow\infty}\frac{\mu(k)}{k^{1+\frac 1n}}\ne 0.\smallskip
\]
\item[(d)]
If $\Gamma_f\subset\mathbb C^{n+m}$ admits the Bernstein type inequalities of exponents $\mu_1$ and $\mu_2$, then it admits such inequalities of all exponents $\mu\ge \min (\mu_1,\mu_2)$.\smallskip
\item[(e)]
If $\Gamma_f\subset\mathbb C^{n+m}$ admits the Bernstein type inequality of exponent $\mu$, then each $\Gamma_{f_w}\subset\mathbb C^{n+m}$, $f_w(z):=f(z+w)$, $z\in\mathbb C^n$, $w\in\mathbb C^n$, admits it as well.\smallskip
\item[(f)]
If graphs $\Gamma_{f_i}\subset\mathbb C^{n_i+m_i}$ of holomorphic maps $f:\mathbb C^{n_i}\rightarrow\mathbb C^{m_i}$
admit Bernstein type inequalities of exponents $\mu_i$, $i=1,2$, then the graph $\Gamma_{f_1\times f_2}\subset \mathbb C^{n_1+n_2+m_1+m_2}$ of the map $(f_1\times f_2)(z_1,z_2):=(f_1(z_1),f_2(z_2))\in\mathbb C^{m_1+m_2}$, $z_i\in\mathbb C^{n_i}$, $i=1,2$, admits the Bernstein type inequality of exponent $\max (\mu_1,\mu_2)$.

\noindent In turn, if $\Gamma_{f_1\times f_2}$ admits the Bernstein type inequality of exponent $\mu$, then each $\Gamma_{f_i}$ admits it as well.
\end{itemize}
\end{Th}
We say that functions $\mu_1,\mu_2:\mathbb Z_+\rightarrow\mathbb R_+$ are {\em equivalent} if there exists a positive real number $c$ such that
\[
\frac 1c \,\mu_1(k)\le \mu_2(k)\le c\,\mu_1(k)\quad {\rm for\ all}\quad k\in\mathbb Z_+.
\]
Let $\mathscr R$ be the set of equivalence classes of functions $\mathbb Z_+\rightarrow\mathbb R_+$. By $\langle\mu\rangle\in \mathscr R$ we denote the equivalence class of $\mu:\mathbb Z_+\rightarrow\mathbb R_+$. We introduce a partial order on $\mathscr R$ writing $\langle\mu_1\rangle\le\langle\mu_2\rangle$ if there exists $c>0$ such that $\mu_1\le c\,\mu_2$. In addition, we regard $\mathscr R$ as an abelian semigroup with addition $\langle\mu_1\rangle+\langle\mu_2\rangle:=\bigl\langle\max(\mu_1,\mu_2)\bigr\rangle$ induced by the pointwise addition of functions.
 Clearly, if $\Gamma_f$ admits the Bernstein type inequality of exponent $\mu$, then it admits such inequality of any equivalent exponent. Therefore it is naturally to consider the set $\mathscr E_f\subset\mathscr R$ of equivalence classes of possible exponents in Bernstein type inequalities for $\Gamma_f$. Then properties (c)--(f) of the theorem can be rephrased as follows:
\begin{itemize}
\item[(c$'$)] 
If $\Gamma_f\subset\mathbb C^{n+m}$ is nonalgebraic, then $\bigl\langle\mu_{{\rm id}}^{1+\frac 1n}\bigr\rangle\in \mathscr R$ is a lower bound of the set $\mathscr E_f$.
\item[(d$'$)] 
$\mathscr E_f$ is a partially ordered subsemigroup of $(\mathscr R,\le, +)$ and every two elements of $\mathscr E_f$ have unique infimum  and  supremum (i.e. $\mathscr E_f$ is a {\em lattice}).
\item[(e$'$)]
$\mathscr E_f$ coincides with $\mathscr E_{f_w}$ for all $w\in\mathbb C^n$.
\item[(f$\,'$)] 
$\mathscr E_{f_1\times f_2}=\mathscr E_{f_1}+\mathscr E_{f_2}$.
 \end{itemize}
 
We say that a function $\mu_o:\mathbb Z_+\rightarrow\mathbb R_+$ is {\em optima}l for $\Gamma_f$ if $\langle\mu_o\rangle$ is the minimal element of $\mathscr E_f$ (in other words,
$\Gamma_f$ admits the Bernstein type inequality of exponent $\mu_o$ and does not admit such inequality of an exponent $\mu$ such that $\mu\le c\mu_0$ for some $c>0$ and $ \varlimsup_{k\rightarrow\infty}\frac{\mu_o(k)}{\mu(k)}=\infty$.) Since $\mathscr E_f$ is a lattice, the minimal element $\langle\mu_0\rangle$ of $\mathscr E_f$ is   also the least element of $\mathscr E_f$, i.e. $\langle\mu_0\rangle\le\langle\mu\rangle$ for all $\langle\mu\rangle\in\mathscr E_f$. Moreover, in this case $\mathscr E_f=\langle\mu_0\rangle+\mathscr R$.

For instance, $\mu_{{\rm id}}$ is optimal for an algebraic $\Gamma_f$. Below we give some other examples of $\Gamma_f$ allowing optimal exponents. The problem of existence of optimal exponents for generic $\Gamma_f$ is open. \smallskip

Let $K\subset\mathbb C^n$ be a nonpluripolar compact set. For  a function $\mu:\mathbb Z_+\rightarrow\mathbb R_+$
we set
\[
u_{K,\mu}^k(z;f):=\sup\left\{\frac{\ln |p_f(z)|}{\max\bigl(1,\mu(k)\bigr)}\, :\, p\in\mathcal P(\mathbb C^{n+m}),\ {\rm deg}\,p= k,\ \sup_{K}|p_f|=1\right\},\  z\in\mathbb C^n,
\]

\[
 u_K^k(r;f):=\max\bigl(1,\mu(k)\bigr)\max_{\|z\|\le r}u_{K,\mu}^k(z;f),\quad r>0,\  k\in\mathbb Z_+.
\]
Approximating polynomials of a given degree by those of a larger one (cf. \eqref{equ3.21} below), one obtains that for each $r>0$ the sequence $u_{K}^k(r;f)$, $k\in\mathbb Z_+$, is nondecreasing. Also, due to the maximum principle for plurisubharmonic functions classes $\bigl\langle u_K^{\cdot}(r;f)\bigr\rangle\in\mathscr R$, $r>0$, form a subsemigroup and a chain $\mathscr U_f^K$. By definition, each element of $\mathscr E_f$ is an upper bound of $\mathscr U_f^K$.

\begin{Th}\label{theo1.3}
\begin{itemize}
\item[(a)] Each $u_{K,\mu}^k$ is a nonnegative continuous plurisubharmonic function on $\mathbb C^n$ equals $0$ on $K$.
\item[(b)]
Graph $\Gamma_f$ admits the Bernstein type inequality of exponent $\mu$ if and only if the (Lebesgue measurable) function
\[
\begin{array}{r}
\displaystyle
u_{K,\mu}(z;f):=\varlimsup_{k\rightarrow\infty}u_{K,\mu}^k(z;f),\quad   z\in\mathbb C^n,
 \end{array}
\]
is locally bounded from above. 
\item[(c)]
An exponent $\mu$ in the Bernstein type inequality for $\Gamma_f$ is optimal if and only if for each subsequence $\bar k=\{k_j\}_{j\in\mathbb N}\subset\mathbb N$ the function
\[
\begin{array}{r}
\displaystyle
u_{K,\mu;\bar k}(z;f):=\varlimsup_{j\rightarrow\infty}u_{K,\mu}^{k_j}(z;f),\quad  z\in\mathbb C^n,
 \end{array}
\]
is not identically $0$. 
\item[(d)] 
An exponent $\mu$ in the Bernstein type inequality for $\Gamma_f$ is optimal if and only if $\langle\mu\rangle\in \mathscr U_f^K$.  In this case $\langle\mu\rangle$ is the maximal element of $\mathscr U_f^K$.
\end{itemize}
\end{Th}
\begin{R}\label{rem1.3}
{\rm (1) Theorem \ref{theo1}\,(f) implies that if functions $\mu_i:\mathbb Z_+\rightarrow\mathbb R_+$ are optimal for $\Gamma_{f_i}$, $i=1,2$, then the function
$\max(\mu_1,\mu_2)$ is optimal for $\Gamma_{f_1\times f_2}$.\smallskip

\noindent (2) For a nonpolynomial entire function $f:\mathbb C\rightarrow\mathbb C^m$ and $p\in\mathcal P(\mathbb C^{m+1})$ by $n_{p_f}(r)$ we denote the number of zeros (counted with their multiplicities) of the univariate holomorphic function $p_f$ in the closed disk $\bar{\mathbb D}_r:=\{z\in\mathbb C\, :\, |z|\le r\}$. (We set $n_{p_f}=-\infty$ if $p_f=0$.) Let 
\[
N^k(r;f):=\sup\bigl\{n_{p_f}(r)\, :\, p\in\mathcal P(\mathbb C^{m+1}),\  {\rm deg}\, p\le k\bigr\}.
\]
The integer-valued function $N^k(\,\cdot\, ;f)$ is nonnegative locally bounded from above and satisfies for all $r>1$ (see \cite[Cor.\,2.3]{CP1}),
\[
\frac{1}{\ln r+16}\,u_{\bar{\mathbb D}_1}^k\bigl(\mbox{$\frac r3$}\bigr)\le N^k(r;f)\le  2u_{\bar{\mathbb D}_1}^k(3r).
\]
Thus, the classes $\langle N^{\cdot}(r;f)\rangle\in\mathscr R$, $r>0$, form
a subsemigroup and a chain $\mathscr N_f$ such that
\[
\bigl\langle u_{\bar{\mathbb D}_1}^\cdot\bigl(\mbox{$\frac r3$}\bigr)\bigr\rangle\le \langle N^\cdot (r;f)\rangle\le  \langle u_{\bar{\mathbb D}_1}^\cdot (3r)\rangle\quad {\rm for\ all}\quad r>1.
\]
In particular, Theorem \ref{theo1.3}\,(b),(d) implies that $\Gamma_f\subset\mathbb C^{m+1}$ admits the Bernstein type inequality of an exponent $\mu$ if and only if $\langle\mu\rangle\in\mathscr R$ is an upper bound of $\mathscr N_f$. In addition, such $\mu$ is optimal for $\Gamma_f$ if and only if $\langle\mu\rangle\in\mathscr N_f$. In this case, $\langle\mu\rangle\in\mathscr R$ is the maximal element of $\mathscr N_f$ so that as an optimal exponent one can take, e.g., the function $N^{\cdot}(r_0;f):\mathbb Z_+\rightarrow\mathbb Z_+$ for a sufficiently large $r_0$.
}
\end{R}
\subsection{} In this section we describe some classes of graphs $\Gamma_f$ admitting Bernstein type inequalities of exponents of polynomial growth.  \smallskip

First, we show that power functions $\mu_{{\rm id}}^{d}(k):=k^d$, $k\in\mathbb Z_+$, $d\in\mathbb N$, are optimal exponents in Bernstein type inequalities on some graphs $\Gamma_f$. \smallskip

In what follows, for holomorphic maps $f_j:\mathbb C^{n_j}\rightarrow \mathbb C^{m_j}$, $1\le j\le l$, by
$f_1\times\cdots\times f_l:\mathbb C^{n_1+\cdots + n_l}\rightarrow\mathbb C^{m_1+\cdots +m _l}$ we denote a map given by the formula
\begin{equation}\label{product}
(f_1\times\cdots\times f_l)(z_1,\dots, z_l):=\bigl(f_1(z_1),\dots, f_l(z_l)\bigr),\qquad z_j\in\mathbb C^{n_j},\quad 1\le j\le l.
\end{equation}

A map $f:\mathbb C\rightarrow\mathbb C^m$ is said to be {\em exponential of maximal transcendence degree} if there are
linearly independent over $\mathbb Q$ complex numbers $\alpha_1,\dots, \alpha_m$ such that\begin{equation}\label{expo}
f(z):=\bigl(e^{\alpha_1 z},\dots, e^{\alpha_m z}\bigr),\quad z\in\mathbb C.
\end{equation}
(In this case the coordinates of $f$ are algebraically independent over the field of rational functions on $\mathbb C$ and so the {\em Zariski closure} of $\Gamma_f\subset\mathbb C^{m+1}$ coincides with $\mathbb C^{m+1}$.)

Let $f_j:\mathbb C\rightarrow\mathbb C^{m_j}$, $1\le j\le l$, be exponential maps of maximal transcendence degrees and $P,Q$ be holomorphic polynomial automorphisms of $\mathbb C^l$  and $\mathbb C^{m_1+\cdots +m_l}$, respectively. We set 
\begin{equation}\label{compos}
\bar m:=\max_{1\le j\le l} m_j\quad {\rm and}\quad F_{P,Q}:=Q\circ (f_1\times\cdots\times f_l)\circ P.
\end{equation}
(The coordinates of the map $F_{P,Q}:\mathbb C^l\rightarrow\mathbb C^{m_1+\cdots +m_l}$ are functions of the form $\sum_{j=1}^J p_j e^{q_j}$, $p_j,q_j\in\mathcal P(\mathbb C^l)$, $1\le j\le J$, called the {\em generalized exponential polynomials} on $\mathbb C^l$.)
\begin{Th}\label{theo1.4}
Graph $\Gamma_{F_{P,Q}}\subset\mathbb C^{l+m_1+\cdots + m_l}$ admits the Bernstein type inequality of optimal exponent $\mu_{{\rm id}}^{\bar m+1}$.
\end{Th}

Our next result reveals the basic property of Bernstein type inequalities on the graphs of nonpolynomial entire functions.
\begin{Th}\label{theo1.5}
Let $f$ be a nonpolynomial entire function on $\mathbb C^n$. Then its graph $\Gamma_f\subset\mathbb C^{n+1}$ admits the Bernstein type inequality of an exponent $\mu$  such that
\begin{equation}\label{eq1.6}
1+\frac 1n \le\varliminf_{k\rightarrow\infty}\frac{\ln\mu(k)}{\ln k}\le 2.
\end{equation}
\end{Th}
It is unclear whether this result is sharp as currently there are no examples of graphs $\Gamma_f\subset\mathbb C^n$, $n\ge 2$, admitting Bernstein type inequalities of exponents $\mu$ for which the corresponding limit in \eqref{eq1.6} is strictly less than two.\smallskip

In our subsequent formulations we use the following definitions and notation.\smallskip

By $\mathbb B_r^n\subset\mathbb C^n$ we denote the open Euclidean ball of radius $r$ centered at $0$; we set
$\mathbb B^n:=\mathbb B_1^n$, $ \mathbb D_r:=\mathbb B_r^1$ and $\mathbb D:=\mathbb B^1$.

For a continuous function $f:\mathbb B_r^n\rightarrow\mathbb C$ we define
\[
M_f(r):=\sup_{\mathbb B_r^n} |f|,\qquad m_f(r):=\ln M_f(r).
\]

Next, recall that an entire function $f$ on $\mathbb C^n$ is of order $\rho_f\ge  0$ if
\[
\rho_f= \varlimsup_{r\rightarrow\infty}\frac{\ln m_f(r)}{\ln r}.
\] 
If $\rho_f <\infty$, then $f$ is called of finite order.

For a nonconstant entire function $f$ on $\mathbb C^n$ of order $\rho_f$ we set
\[
\phi_f(t):=m_f(e^t),\quad t\in\mathbb R.
\]
Then $\phi_f$ is a convex increasing function. 

By $\mathscr C$ we denote the class of nonpolynomial entire functions $f$ satisfying one of the following conditions:
\begin{itemize}
\item[(I)]
If $\rho_f<\infty$,
\begin{equation}\label{cI}
\varlimsup_{t\rightarrow\infty}\frac{\phi_f(t+1)-\phi_f(t)}{\phi_f(t)-\phi_f(t-1)}<\infty .
\end{equation}
\item[(II)]
If $\rho_f=\infty$, 
\begin{equation}\label{cII}
\lim_{t\rightarrow\infty}t^2\left(\frac{1}{\ln\psi(t)}\right)'=0,
\end{equation}
where $\psi\in C(\mathbb R)$ is a convex increasing function such that
\[
\lim_{t\rightarrow\infty}\frac{\ln\psi(t)}{\ln\phi_f(t)}=1.
\]
\end{itemize}
\begin{R}
{\rm (1) Each convex function is differentiable at all but at most countably many points. Thus the limit in \eqref{cII} is taken over the domain of $\psi'$.\smallskip

\noindent (2) Conditions \eqref{cI} and \eqref{cII} are complimentary to each other (i.e. there are no entire functions satisfying both of these conditions). 
}
\end{R}

Let $f_j:\mathbb C^{n_j}\rightarrow\mathbb C$, $1\le j\le m$, be functions of class $\mathscr C$ and $P,Q$ be holomorphic polynomial automorphisms of  $\mathbb C^{n_1+\cdots +n_m}$ and $\mathbb C^m$, respectively. 
\begin{Th}\label{theo1.7}
The graph
$\Gamma_{F_{P,Q}}\subset\mathbb C^{n_1+\cdots +n_m+m}$ of the map $F_{P,Q}:=Q\circ (f_1\times\dots\times f_m)\circ P:\mathbb C^{n_1+\cdots+n_m}\rightarrow\mathbb C^m$ admits the Bernstein type inequality of exponent $\mu(k):=k^{2+\varepsilon (k)}$, $k\in\mathbb Z_+$, where $\varepsilon= 0$ if all $\rho_{f_j}<\infty$ and $\varepsilon\ge 0$ decreases to $0$ if one of $\rho_{f_j}=\infty$.

Moreover, if all $n_j=1$ and all $\rho_{f_j}<\infty$, then $\mu_{{\rm id}}^2$ is an optimal exponent for $\Gamma_{F_{P,Q}}$.
\end{Th}

Our next result describes some class of curves $\Gamma_f\subset\mathbb C^{m+1}$ admitting  Bernstein type inequalities of exponents of polynomial growth.
Up till now, for $m\ge 2$ the only known examples of such curves were graphs of holomorphic maps $f:\mathbb C\rightarrow \mathbb C^m$ with coordinates being exponents of polynomials. As follows from the results established in  \cite{BBL}  graphs of such maps admit Bernstein type inequalities of exponents $\mu_{\rm id}^{3m+3}$. 
\begin{Th}\label{theo1.8}
Suppose that nonpolynomial entire functions $f_j:\mathbb C\rightarrow\mathbb C$, $1\le j\le m$, of class $\mathscr C$ are such that $\rho_{f_1}\le\cdots\le\rho_{f_m}$ and
\begin{equation}\label{eq1.4}
\lim_{r\rightarrow\infty}\frac{m_{f_j}(r)-m_{f_j}\bigl(\frac re\bigr)}{\sqrt{m_{f_{j+1}}(r)-m_{f_{j+1}}\bigl(\frac{r}{e}\bigr)}}=0\qquad {\rm if}\quad \rho_{f_{j+1}}<\infty,\quad
j\in\{1,\dots, m-1\},
\end{equation}
and
\begin{equation}\label{eq1.5}
\varlimsup_{r\rightarrow\infty}\frac{\ln m_{f_j}(r)}{\ln m_{f_{j+1}}(\frac re)}<\frac 12\qquad {\rm if}\quad \rho_{f_j}=\infty,\quad  j\in \{1,\dots, m-1\}.
\end{equation}
Then for $f=(f_1,\dots, f_{m}):\mathbb C\rightarrow\mathbb C^{m}$ its graph $\Gamma_f\subset\mathbb C^{m+1}$ admits the Bernstein type inequality of exponent $\mu(k):=k^{2^{m}+\varepsilon(k)}$, $k\in\mathbb Z_+$, for some $\varepsilon :
\mathbb Z_+\rightarrow \mathbb R_+$ decreasing to zero. Here $\varepsilon= 0$ if all $\rho_{f_j}<\infty$.
\end{Th}
\begin{R}\label{remark1.9}
{\rm (1) The Zariski closure of $\Gamma_f$ is $\mathbb C^{m+1}$, (i.e. each holomorphic polynomial vanishing on $\Gamma_f$ is zero), cf. Remark \ref{rem6.6} in section 6.3.\smallskip

\noindent (2) According to the Jensen type inequality, see \cite[Lm.\,1]{VP}, for each $p\in\mathcal P(\mathbb C^{m+1})$ the number of zeros (counted with their multiplicities) of the function $p_f(z):=p(z,f_1(z),\dots, f_m(z))$, $z\in\mathbb C$, in the closed disk $\bar{\mathbb D}_r$, is bounded from above by $C(r)({\rm deg}\,p)^{2^m+1}$ for some positive constant $C(r)$, $r>0$.\smallskip

\noindent (3) Let $f^j=(f_1^j,\dots, f_{m_j}^j):\mathbb C\rightarrow\mathbb C^{m_j}$, $m_j\in\mathbb N$, $1\le j\le l$, be holomorphic maps satisfying conditions of Theorem \ref{theo1.8}. Let
$P$ and $Q$ be  holomorphic polynomial automorphisms of $\mathbb C^l$ and $\mathbb C^{m_1+\cdots +m_l}$, respectively. We set 
\[
F_{P,Q}:=Q\circ (f^1\times\cdots\times f^l)\circ P:\mathbb C^l\rightarrow \mathbb C^{m_1+\cdots +m_l}\quad {\rm and}\quad \bar{m}:=\max_{1\le j\le l}m_j.
\]
Then the graph $\Gamma_{F_{P,Q}}\subset\mathbb C^{l+m_1+\cdots m_l}$ of $F_{P,Q}$ satisfies the Bernstein type inequality of exponent $\mu(k):=k^{2^{\bar{m}}+\varepsilon(k)}$, $k\in\mathbb Z_+$, for some $\varepsilon :
\mathbb Z_+\rightarrow \mathbb R_+$ decreasing to zero. Here $\varepsilon= 0$ if all $\rho_{f_i^j}<\infty$.

\noindent The proof of this result follows from Theorems \ref{theo1.8} and \ref{theo1}\,(f) by means of the arguments of the proof of Theorem \ref{theo1.7} in section 6.2 (cf. also Lemma \ref{lem4.1}).
}
\end{R}

We illustrate  the theorem by a simple example (see section 1.4 for other examples).
\begin{E}\label{ex1.9}
{\rm By $e^{\circ p}:\mathbb C\rightarrow \mathbb C$, $p\in\mathbb N$, we denote the $p$ times composition of the exponential function with itself. For some $m_1,\dots, m_l\in\mathbb N$, we set $m:=m_1+\cdots +m_l$. Consider the following univariate entire functions $f_1,\dots, f_m$.

If $1\le j\le m_1$, then $f_j(z)=e^{z^{n_j}}$, $z\in\mathbb C$, $n_j\in\mathbb N$, where
\[
\frac{n_{j+1}}{n_j}>2\quad {\rm for\ all}\quad 1\le j\le m_1-1.
\]

If $m_1+1\le j\le m_2$, then $f_j(z)=e^{\circ 2}(n_j z)$, $z\in\mathbb C$, $n_j>0$, where
\[
\frac{n_{j+1}}{n_j}>2e\quad {\rm for\ all}\quad m_1+1\le j\le m_2-1.
\]

If $m_{k-1}\le j\le m_{k}$ for $3\le k\le l$, then $f_j(z)=e^{\circ k}(n_j z)$, $z\in\mathbb C$,  $n_j>0$, where
\[
\frac{n_{j+1}}{n_j}>e\quad {\rm for\ all}\quad m_{k-1}+1\le j\le m_k-1.
\]
One easily checks that 
\[
m_{f_j}(r)=\ln f_j(r)\quad {\rm for\ all}\quad 1\le j\le m,
\]
and that all $f_j\in\mathscr C$ and satisfy the hypotheses of Theorem \ref{theo1.8}. Hence, for $f=(f_1,\dots, f_{m}):\mathbb C\rightarrow\mathbb C^{m}$ its graph $\Gamma_f\subset\mathbb C^{m+1}$ admits the Bernstein type inequality of exponent $\mu(k):=k^{2^{m}+\varepsilon(k)}$, $k\in\mathbb Z_+$, for some $\varepsilon :
\mathbb Z_+\rightarrow \mathbb R_+$ decreasing to zero.
}
\end{E}
\subsection{} In this section we formulate some properties of functions of class $\mathscr C$.\smallskip

The first three properties follow straightforwardly from the definition of class $\mathscr C$.
\begin{Proposition}\label{propo1.11}
\begin{itemize}
\item[(1)]
If $f\in\mathscr C$ and $g$ is an entire function such that 
\[
0<\varliminf_{r\rightarrow\infty}\frac{M_g(r)}{ M_f(r)}\le\varlimsup_{r\rightarrow\infty}\frac{M_g(r)}{ M_f(r)}<\infty,
\] 
then $g\in\mathscr C$ as well.\smallskip

\item[(2)]
If $f\in\mathscr C$, $\rho_f=\infty$,  and $g$ is an entire function such that 
\[
0<\varliminf_{r\rightarrow\infty}\frac{m_g(r)}{ m_f(r)}\le\varlimsup_{r\rightarrow\infty}\frac{m_g(r)}{ m_f(r)}<\infty,
\] 
then $g\in\mathscr C$ as well.\smallskip

\item[(3)] If $f\in\mathscr C$, then $f^n\in\mathscr C$ for all $n\in\mathbb N$.
\end{itemize}
\end{Proposition}
Properties (1) and (3) imply that if $f\in\mathscr C$ and $p\in\mathcal P(\mathbb C)$ is a nonconstant polynomial, then $p\circ f\in\mathscr C$.\medskip

A function $\rho\in C^1(0,\infty)$ satisfying 
conditions
\[
\lim_{r\rightarrow\infty}\rho(r)=\rho_f\quad {\rm and}\quad \lim_{r\rightarrow\infty} r\rho'(r)\ln r=0
\]
is called the  {\em proximate order} of an entire function  $f$ if
\[
\varlimsup_{r\rightarrow\infty}\frac{m_f(r)}{r^{\rho(r)}}=:\sigma_f\in (0,\infty).
\]
It is well known that an entire function of finite order has a proximate order, see, e.g. \cite[Th.\,I.16]{L}.
\begin{Proposition}\label{prop1.9}
If an entire function of finite positive order $f$ has a proximate order $\rho$ satisfying condition
\[
\varliminf_{r\rightarrow\infty}\frac{m_f(r)}{r^{\rho(r)}}>0,
\]
then $f\in\mathscr C$.
\end{Proposition}
\begin{E}\label{ex1.10}
{\rm Let
\[
f=\sum_{j=1}^l p_j\, e^{q_j},\quad p_j,\, q_j\in\mathcal P(\mathbb C^n),\ 1\le j\le l,
\]
be the generalized exponential polynomial. (We assume that $f\not\in \mathcal P(\mathbb C^n)$.) Let
$\hat q_j$ denote the homogeneous component of degree ${\rm deg}\, q_j$ of $q_j$ (i.e. ${\rm deg}\, (q_j-\hat q_j)<\deg\, q_j$). 
Then
\[
\rho_f=\max_{1\le j\le l}{\rm deg}\, q_j\quad {\rm and }\quad \sigma_f=\lim_{r\rightarrow\infty}\frac{m_f(r)}{r^{\rho_f}}=\max_{j\,:\, {\rm deg}\, q_j=\rho_f}\left\{\sup_{\mathbb B^n}|\hat q_j|\right\}.
\]
Thus, due to Proposition \ref{prop1.9},  $f\in\mathscr C$. 

Further, if $g$ is an entire function such that $\displaystyle \varlimsup_{r\rightarrow\infty}\mbox{$\frac{m_g(r)}{r^{\rho_f}}$}<\sigma_f$, then 
\[
\varlimsup_{r\rightarrow\infty}\frac{M_{g}(r)}{ M_f(r)}<1.
\]
Hence, by Proposition \ref{propo1.11}\,(1), $f+g\in\mathscr C$ as well.
}
\end{E}

To formulate other properties consider the subclass of $\mathscr C$ of nonpolynomial entire functions satisfying condition
\begin{equation}\label{eq1.11}
\lim_{t\rightarrow\infty}\frac{\phi_f(t)}{t^2}=\infty.
\end{equation}
It is easily seen that each $f\in\mathscr C$ with $\rho_f=\infty$ satisfies \eqref{eq1.11}, see Lemma \ref{lemma7.2} below.
\begin{Proposition}\label{prop1.10}
If $f\in\mathscr C$ satisfies condition \eqref{eq1.11}, then  functions $e^f,\sin f,\cos f\in\mathscr C$.
In addition, if $f$ is univariate, then its derivative and every antiderivative are of class $\mathscr C$ and satisfy \eqref{eq1.11}.
\end{Proposition}
\begin{R}
{\rm It is worth noting that if $f$ is a nonpolynomial entire function, then functions $e^f, \sin f, \cos f$ are of infinite order. Thus under the hypothesis of the theorem they satisfy condition \eqref{cII}.
}
\end{R}
\begin{E}\label{ex1.12}
{\rm Let $f$ and $g$ be as in Example \ref{ex1.10}. Then $\phi_{f+g}$ is equivalent to $\phi_f$. Moreover, $\rho_f\ge 1$ and 
\[
\sigma_f=\lim_{t\rightarrow\infty}\frac{\phi_f(t)}{e^{\rho_f t}}\in (0,\infty).
\]
Hence, $f+g$ satisfies condition \eqref{eq1.11}. In particular, due to Proposition \ref{prop1.10}, functions $e^{f+g}, \sin(f+g), \cos(f+g)\in\mathscr C$.
}
\end{E}

To present more explicit examples of entire functions of class $\mathscr C$, we describe a subclass of univariate functions in $\mathscr C$ satisfying condition \eqref{eq1.11} in terms of the coefficients of their Taylor expansions at $0\in\mathbb C$.

Let $h: \mathbb R_+\rightarrow \mathbb R$ be a continuous increasing function satisfying conditions
\begin{equation}\label{eq1.12}
\lim_{t\rightarrow\infty}\frac{h(t)}{t}=\infty\quad {\rm and}\quad \varlimsup_{t\rightarrow\infty}\frac{h(t+1)}{h(t)}<\infty.
\end{equation}
\begin{Th}\label{theo1.11}
Let 
\[
f(z):=\sum_{j=0}^\infty c_j z^j,\quad z\in\mathbb C,
\]
be a holomorphic function in a neighbourhood of $0\in\mathbb C$ such that
\[
\ln |c_j|=-\int_0^{j} h^{-1}(s)\, ds,\quad j\in\mathbb Z_+.
\]
Then $f$ is an entire function of finite order of class $\mathscr C$ satisfying condition \eqref{eq1.11}. Moreover,
\[
\rho_f=\varlimsup_{t\rightarrow\infty}\frac{\ln h(t)}{t}.
\]
\end{Th}
By $f_h$ we denote an entire function constructed by means of a function $h$  satisfying conditions \eqref{eq1.12} as in the above result. 
\begin{R}\label{rem1.18}
{\rm In section 8.1 we show that for each $c>\rho_{f_h}$ there exists a number $t_c$ such that for all $t\ge t_c$
\[
\int_{h^{-1}(0)}^t h(s)\,ds-t\le \phi_{f_h}(t)\le ct+\int_{h^{-1}(0)}^t h(s)\,ds,
\]
see \cite[Ch.\,I.2]{L}, \eqref{equ8.45},  \eqref{eq8.43}.}
\end{R}
The following result allows to construct new functions of class $\mathscr C$ by means of functions of the form $f_h$.
In its formulation, by $\omega_t(\cdot\, ;g)$, $t\ge 0$, we denote the modulus of continuity of a function $g\in C(a,\infty)$, $a\le 0$, restricted to the interval $[0,t]$.
\begin{Proposition}\label{prop1.15}
Let $h_1,h_2:\mathbb R_+\rightarrow\mathbb R$ be continuous increasing functions satisfying conditions  \eqref{eq1.12} such that
\[
\varliminf_{t\rightarrow\infty}\bigl(h_1(t)-h_2(t)\bigr)>\varlimsup_{t\rightarrow\infty}\frac{\ln h_1(t)}{t}+\varlimsup_{t\rightarrow\infty}\frac{\omega_t(1;h_1^{-1})}{h_1^{-1}(t)}.
\]
Then all entire functions of the form $c_1f_{h_1}+c_2f_{h_2}$, $(c_1,c_2)\in\mathbb C^2\setminus\{(0,0)\}$, are of class $\mathscr C$.
\end{Proposition}
Since $h^{-1}$ is an increasing function, the second term on the right-hand side of the above inequality is bounded from above by one.
\begin{E}\label{ex1.15}
{\rm (a) Let $h(t)=\bigl(\frac{t}{\alpha}\bigr)^{\frac{1}{\alpha-1}}$, $\alpha\in (1,2)$. Then, due to Theorem \ref{theo1.11}, $f_h(z):=\sum_{j=0}^\infty c_j z^j$, $ z\in\mathbb C$, where
\[
\ln |c_j|=-\int_0^j \alpha s^{\alpha-1}\,ds=-j^\alpha,\quad j\in\mathbb Z_+,
\]
is a nonpolynomial entire function of order zero of class $\mathscr C$.

Observe that
\[
\omega_h:=\varlimsup_{t\rightarrow\infty}\frac{\omega_t(1;h^{-1})}{h^{-1}(t)}\le\varlimsup_{t\rightarrow\infty}\frac{\sup_{s\in [0,t]}\bigl(\alpha s^{\alpha-1}\bigr)' }{\alpha t^{\alpha-1}}=\varlimsup_{t\rightarrow\infty}\frac{\alpha-1}{t}=0.
\]
Thus by Proposition \ref{prop1.15} for each $\tilde h\in C(\mathbb R_+)$ satisfying conditions \eqref{eq1.12} and
\[
\varliminf_{t\rightarrow\infty}(h(t)-\tilde h(t))>0,
\]
all functions of the form $c_1f_{h_1}+c_2f_{h_2}$, $(c_1,c_2)\in\mathbb C^2\setminus\{(0,0)\}$, are of class $\mathscr C$. In particular, this is true for
$\tilde h=h-c$, $c>0$. In this case $f_{\tilde h}(z)=\sum_{j=0}^\infty \tilde c_j z^j$, $z\in\mathbb C$, is such that
\[
\ln |\tilde c_j|=-(j+c)^\alpha+ c^\alpha,\quad j\in\mathbb Z_+.
\]
Thus, all entire functions $f$ of the form
\[
f(z)=\sum_{j=0}^\infty \bigl( c_1e^{-j^\alpha +\sqrt{-1}\,\theta_j}+c_2 e^{-(j+c)^\alpha+\sqrt{-1}\,\tilde\theta_j}\bigr)\,z^j,\, \ \theta_j,\tilde\theta_j\in\mathbb R,\  j\in\mathbb Z_+,\  (c_1,c_2)\in\mathbb C^2\setminus\{(0,0)\},
\]
are of class $\mathscr C$.\smallskip

\noindent (b) Let $h(t)=e^{\alpha t-1}-1$, $\alpha>0$. Then Theorem \ref{theo1.11} implies that $f_h(z):=\sum_{j=0}^\infty c_j z^j$, $ z\in\mathbb C$, where
\[
\ln |c_j|=-\int_0^j \frac{\ln(s+1)+1}{\alpha}\,ds=-\frac{(j+1)\ln(j+1)}{\alpha},\quad j\in\mathbb Z_+,
\]
is a nonpolynomial entire function of order $\alpha$ of class $\mathscr C$.

Next, in this case
\[
\omega_h:=\varlimsup_{t\rightarrow\infty}\frac{\omega_t(1;h^{-1})}{h^{-1}(t)}\le\varlimsup_{t\rightarrow\infty}\frac{\sup_{s\in [0,t]}\bigl(\ln(s+1) \bigr)' }{\ln(t+1)}=0.
\]
Thus by Proposition \ref{prop1.15} for each $\tilde h\in C(\mathbb R_+)$ satisfying conditions \eqref{eq1.12} and
\[
\varliminf_{t\rightarrow\infty}(h(t)-\tilde h(t))>\alpha,
\]
all functions of the form $c_1f_{h_1}+c_2f_{h_2}$, $(c_1,c_2)\in\mathbb C^2\setminus\{(0,0)\}$, are of class $\mathscr C$. For instance, this is true for
$\tilde h=h-c$, $c>\alpha$. Then $f_{\tilde h}(z)=\sum_{j=0}^\infty \tilde c_j z^j$, $z\in\mathbb C$, is such that
\[
\ln |\tilde c_j|=-\frac{(j+1+c)\ln(j+1+c)}{\alpha}+\frac{(c+1)\ln(c+1)}{\alpha},\quad j\in\mathbb Z_+.
\]
Thus, all entire functions $f$ of the form
\[
f(z)=\sum_{j=0}^\infty \left(\frac{c_1e^{\sqrt{-1}\,\theta_j}}{(j+1)^{\frac{j+1}{\alpha}}}+\frac{c_2e^{\sqrt{-1}\,\tilde\theta_j}}{(j+1+c)^{\frac{j+1+c}{\alpha}}}\right)z^j,\, \ \theta_j,\tilde\theta_j\in\mathbb R,\  j\in\mathbb Z_+,\  (c_1,c_2)\in\mathbb C^2\setminus\{(0,0)\},
\]
are of class $\mathscr C$.
}
\end{E}
Now, as the corollary of Theorems \ref{theo1.11} and \ref{theo1.8} we obtain:
\begin{C}\label{cor1.15}
Suppose $h_1,\dots, h_m\in C(\mathbb R_+)$  are increasing functions satisfying condition \eqref{eq1.12}  such that for some integer $1\le l\le m$
\begin{equation}\label{eq1.13}
\lim_{t\rightarrow\infty}\frac{h_{j}(t)}{\sqrt{h_{j+1}(t)}}=0\quad {\rm for\ all}\quad  1\le j\le l-1
\end{equation}
and
\begin{equation}\label{eq1.14}
\varlimsup_{t\rightarrow\infty}\frac{\int_{M}^{t} h_{j}(s)\,ds}{\int_{M}^{t-1} h_{j+1}(s)\,ds}<\frac 12 \quad {\rm for\ all}\quad  l+1\le j\le m-1,
\end{equation}
where 
\[
M:=\max_{l+1\le j\le m}\{h_j^{-1}(0)\}.
\]
Then entire functions $f_{h_1},\dots, f_{h_l}, e^{f_{h_{l+1}}},\dots, e^{f_{h_m}}$ satisfy the hypotheses of Theorem \ref{theo1.8}. Therefore for $f=(f_{h_1},\dots, f_{h_l}, e^{f_{h_{l+1}}},\dots, e^{f_{h_m}}):\mathbb C\rightarrow\mathbb C^{m}$ its graph $\Gamma_f\subset\mathbb C^{m+1}$ admits the Bernstein type inequality of exponent $\mu(k):=k^{2^{m}+\varepsilon(k)}$, $k\in\mathbb Z_+$, for some $\varepsilon :
\mathbb Z_+\rightarrow \mathbb R_+$ decreasing to zero. Here $\varepsilon= 0$ if $l=m$.
\end{C}
\begin{R}\label{rem1.22}
{\rm Note that condition
\[
\varlimsup_{t\rightarrow\infty}\frac{h_{j}(t+1)}{h_{j+1}(t)}<\frac 12 \quad {\rm for\ all}\quad  l+1\le j\le m-1
\]
implies \eqref{eq1.14}.
}
\end{R}
\begin{E}\label{eq1.19}
{\rm Let $l_1<l<m_1<m$ be some natural numbers.
Suppose that $h_j(t)=\left(\frac{t}{\alpha_j}\right)^{\frac{1}{\alpha_j-1}}$, $\alpha_j\in (1,2)$, for $1\le j\le m_1$ and $m_2+1\le j\le m_3$, and $h_j(t)=e^{\alpha_j t-1}-1$, $\alpha_j>0$, for $m_1+1\le j\le l\le m_2$ and $m_3+1\le j\le m_4$, where
\[
\begin{array}{l}
{\rm (i)}\quad 2\alpha_{j+1}<\alpha_j+1\quad {\rm for\ all}\quad 1\le j\le l_1;\quad \ {\rm (ii)}\quad
\alpha_{j}<2\alpha_{j+1}\quad {\rm for\ all}\quad l_1+1\le j\le l;\medskip\\
{\rm (iii)}\quad \alpha_{j+1}<\alpha_j\quad {\rm for\ all}\quad l+1\le j\le m_1; \quad {\rm (iv)}\quad
\alpha_{j}<\alpha_{j+1}\quad {\rm for\ all}\quad m_1+1\le j\le m.
\end{array}
\]
Then entire functions $f_{h_1},\dots, f_{h_{l}}, e^{f_{h_{l+1}}},\dots, e^{f_{h_{m}}}$ satisfy assumptions of Corollary \ref{cor1.15}.
}
\end{E}

\section{Proof of Theorem \ref{theo1}}
For basic facts of complex algebraic geometry, see, e.g., the book \cite{M}.

(a) Without loss of generality we may assume that $\Gamma_f$ is 
nonalgebraic. Let $\mathcal I_f\subset\mathcal P(\mathbb C^{n+m})$ be the ideal of holomorphic polynomials vanishing on $\Gamma_f$ and $\mathcal Z_f\subset\mathbb C^{n+m}$ be the set of zeros of $\mathcal I_f$. Let $X_f$ be the irreducible component of $\mathcal Z_f$ containing $\Gamma_f$.
Then $X_f$ is the complex algebraic subvariety of $\mathbb C^{n+m}$ of pure dimension $l\ge n+1$. Let  $U\Subset X_f$ be a relatively compact open subset such that 
\[
U\cap\Gamma_f=\{(z, f(z))\in\Gamma_f\, :\, z\in\mathbb B^n\}.
\]
Since $U$ is a nonpluripolar subset of  $\mathcal Z_f$, \cite[Th.~2.2]{S} implies that for each $r\ge 1$, there exists a constant $A(r)$ such that for all $p\in\mathcal P(\mathbb C^{n+m})$,
\begin{equation}\label{eq2.6}
\sup_{\mathbb B^n_{re}} |p_f|\le A(r)^{{\rm deg}\,p}\,\sup_{U}|p|,\quad {\rm where}\quad  p_f:=p(\cdot,f(\cdot)).
\end{equation}

Next, we prove
\begin{Lm}\label{lem2.1}
For each $k\in\mathbb Z_+$ there exists a positive constant $c(k)$ such that for all $p\in\mathcal P(\mathbb C^{n+m})$ with ${\rm deg}\,p\le k$,
\[
\sup_U |p|\le c(k)\sup_{U\cap\Gamma_f}|p|.
\]
\end{Lm}
\begin{proof}
Assume, on the contrary, that the statement is wrong for some $k_0\in\mathbb Z_+$. Then there exists the sequence of polynomials $\{p_i\}_{i\in\mathbb N}\subset\mathcal P(\mathbb C^{n+m})$ of degrees $\le k_0$ such that
\[
\sup_U |p_i|=1\quad {\rm for\ all}\quad i\in\mathbb N\quad {\rm and}\quad \lim_{i\rightarrow\infty}\,\sup_{U\cap\Gamma_f}|p_i|=0.
\]
Due to \cite[Th.~2.2]{S} the sequence $\{p_i\}_{i\in\mathbb N}$ is uniformly bounded on each compact subset of $X_f$ (cf. \eqref{eq2.6}). Thus, due to the Montel theorem, $\{p_i\}_{i\in\mathbb N}$ contains a subsequence uniformly converging on compact subsets of $X_f$ to a function $g\in C(X_f)$ holomorphic outside of the set of singular points of $X_f$ and such that $\sup_U |g|=1$ and $g|_{\Gamma_f}=0$. By definition, $g$ is a regular function on the affine algebraic variety $X_f$ of pure dimension $l$. Hence, there exist polynomials $q_1,\dots , q_s\in\mathcal P(\mathbb C^{n+m})$, where $q_s|_{X_f}\not\equiv 0$, such that
\begin{equation}\label{eq2.7}
g^s(x)+q_1(x) g^{s-1}(x)+\cdots + q_s(x)=0\quad {\rm for\ all}\quad x\in X_f.
\end{equation}
Equation \eqref{eq2.7} and the fact that $g|_{\Gamma_f}=0$ imply that $q_s=0$ on $\Gamma_f$. Therefore $q_s|_{X_f}=0$ by the definition of $X_f$, a contradiction proving the lemma.
\end{proof}
We set
\[
C(r)=e A(r),\quad r\ge 1,\quad \mu(k):=\big\lfloor\max\{\ln c({\rm deg}\,p), {\rm deg}\,p\}\big\rfloor +1,\quad k\in\mathbb Z_+.
\]
Then using the lemma and equation \eqref{eq2.6} we obtain, for each $r\ge 1$  and all $p\in\mathcal P(\mathbb C^{n+m})$, 
\begin{equation}\label{eq2.8}
\sup_{\mathbb B^n_{re}} |p_f|\le c({\rm deg}\,p)A(r)^{{\rm deg}\, p}\sup_{\mathbb B^n} |p_f|\le C(r)^{\mu({\rm deg}\,p)}\sup_{\mathbb B_r^n} |p_f|.
\end{equation}
Inequality \eqref{eq2.8} and the Hadamard three circle theorem imply (see, e.g., \cite[Sect.~3.1]{B}),
 for each $0<r<1$ and all $p\in\mathcal P(\mathbb C^{n+m})$, 
 \[
\sup_{\mathbb B^n_{re}} |p_f|\le C(1)^{\mu({\rm deg}\,p)}\sup_{\mathbb B_r^n} |p_f|.
\]
Thus, $\Gamma_f$ admits the Bernstein type inequality of the exponent
$\mu$.\smallskip

(b) Suppose that condition \eqref{e2}  is valid for a compact nonpluripolar set $K\subset\mathbb B_{r_0}^n$ for some $r_0>0$. Then for each $r\ge r_0$  and all $p\in\mathcal P(\mathbb C^{n+m})$, $p|_{\Gamma_f}\not\equiv 0$, 
\[
\frac{\sup_{\mathbb B^n_{er}} |p_f|}{\sup_{\mathbb B^n_r} |p_f|}\le \frac{\sup_{\mathbb B^n_{er}} |p_f|}{\sup_{\mathbb B^n_{r_0}} |p_f|}\le \frac{\sup_{\mathbb B^n_{er}} |p_f|}{\sup_{K} |p_f|}\le C(K;r)^{\mu({\rm deg}\,p)}.
\]
For $r< r_0$ a similar inequality with $C(K;r)$ replaced by $C(K;r_0)$ follows from that for $r=r_0$ by the Hadamard three circle theorem. Thus, $\Gamma_f$ admits the Bernstein type inequality of exponent $\mu$.

Conversely, assume that $\Gamma_f$ admits the Bernstein type inequality of exponent $\mu$. Suppose that $K\Subset\mathbb B^n_{r_0}$, $r_0>0$, is a compact nonpluripolar set. We set
\[
t_K:=\inf\left\{\frac{1}{\mu({\rm deg}\,p)}\ln \frac{\sup_K |p_f|}{\sup_{\mathbb B^n_{r_0}}|p_f|}   \right\},
\]
where the infimum is taken over all polynomials $p\in \mathcal P(\mathbb C^{n+m})$ such that $p|_{\Gamma_f}\not\equiv 0$.
\begin{Lm}\label{lem2.2}
$t_K>-\infty$.
\end{Lm}
\begin{proof}
Assume, on the contrary, that $t_K=-\infty$. Then there
exists a sequence of nonidentical zero on $\Gamma_f$ polynomials $p_k\in \mathcal P(\mathbb C^{n+m})$ such that
\[
\lim_{k\rightarrow\infty}\frac{1}{\mu({\rm deg}\,p_{k })}\ln \frac{\sup_K |p_{k f}|}{\sup_{\mathbb B^n_{r_0}}|p_{kf}|}=-\infty. 
\]
Let us consider the function
\[
u(z):=\varlimsup_{k\rightarrow\infty}\frac{1}{\mu({\rm deg}\,p_{k})}\ln \frac{|p_{kf}(z)|}{\sup_{\mathbb B^n_{r_0}}|p_{kf}|},\qquad z\in\mathbb C^n.
\]
Then the Bernstein type inequality \eqref{e1} implies that for each $r\ge r_0$ there exists a real number $\tilde c(r)$ such that
\[
\sup_{\mathbb B_r^n}u\le \tilde c(r).
\]
Let $u^*$ be the upper semicontinuous regularization of $u$. The previous inequality and the Hartogs lemma on subharmonic functions imply that $u^*$ is a nonidentical $-\infty$ plurisubharmonic function on $\mathbb C^n$ such that
$\sup_{\mathbb B_{r_0}^n} u^*=0$. Moreover,  $u|_{K}=-\infty$ and the set $S\subset\mathbb C^n$ where $u$ differs from $u^*$ is pluripolar, see \cite[Th.\,4.2.5]{BT}. Since by the hypothesis $K$ is nonpluripolar, $K\setminus S$ is nonpluripolar as well. Thus $u^*=-\infty$ on the nonpluripolar set $K\setminus S$ and so it equals $-\infty$ everywhere, a contradiction proving the lemma.
\end{proof}
Lemma \ref{lem2.2}  and the Bernstein type inequality show that for all $p\in\mathcal P(\mathbb C^{n+m})$, $r\ge r_0$,
\[
\begin{array}{r}
\displaystyle
\sup_{\mathbb B_{r}^n}|p_f|\le C(re^{-1})^{\mu({\rm deg}\,p)}\sup_{\mathbb B_{re^{-1}}^n}|p_f|\le\cdots\le \left(\prod_{i=1}^{\big\lfloor\ln\frac{r}{r_0}\big\rfloor+1} C(re^{-i})^{\mu({\rm deg}\,p)}\right)\sup_{\mathbb B_{r_0}^n}|p_f|\medskip\\
\displaystyle \le
C(K;r)^{\mu({\rm deg}\,p)}\sup_{K}|p_f|,\qquad\qquad\ \ \
\end{array}
\]
where
\[
C(K;r):=e^{-t_K}\left(\prod_{i=1}^{\big\lfloor\ln\frac{r}{r_0}\big\rfloor+1} C(re^{-i})\right).
\]
Also, for $0<r<r_0$ we obviously have
\[
\sup_{\mathbb B_{r}^n}|p_f|\le\sup_{\mathbb B^n}|p_f|\le
(e^{-t_K})^{\mu({\rm deg}\,p)}\sup_{K}|p_f|.
\]
This completes the proof of (b).
\smallskip

(c) By $\mathcal P_k (\mathbb C^{N})\subset \mathcal P (\mathbb C^{N})$ we denote the space of holomorphic polynomials of degree at most $k$. Then
\[
{\rm dim}_{\mathbb C}\bigl(\mathcal P_k (\mathbb C^{N})\bigr)={ N+k \choose N}=:d_{k,N}.
\]

Assume without loss of generality that the coordinate $f_1:\mathbb C^n\rightarrow\mathbb C$ of the holomorphic map $f=(f_1,\dots, f_n):\mathbb C^n\rightarrow\mathbb C^m$ is nonpolynomial  (for otherwise, if all $f_i$ are polynomials, then $\Gamma_f$ is algebraic). 

In what follows for an entire function $g$ on $\mathbb C^n$ by
$
\sum_{|\alpha|=0}^\infty\, [g]_\alpha z^\alpha$
we denote its Taylor series at $0\in\mathbb C^n$. Here $\alpha=(\alpha_1,\dots,\alpha_n)\in\mathbb Z_+^n$, $|\alpha|=\alpha_1+\cdots +\alpha_n$  and $z^\alpha:=z_1^{\alpha_1}\cdots z_n^{\alpha_n}$, $z=(z_1,\dots, z_n)\in\mathbb C^n$. 

We set
\begin{equation}\label{eq2.10}
s_k:=\left\lfloor\frac{k^{1+\frac 1n}}{(n+2)^{\frac 1n}}\right\rfloor.
\end{equation}
Since
\[
\lim_{k\rightarrow\infty}\frac{d_{s_k, n}}{k^{n+1}}=\frac{1}{(n+2)n!}<\frac{1}{(n+1)!}=\lim_{k\rightarrow\infty}\frac{d_{k,n+1}}{k^{n+1}},
\]
there is $k_0\in\mathbb N$ such that $d_{s_k,n} <d_{k,n+1}$ for all $k \ge k_0$. 
\begin{Lm}\label{lem2.3}
For every $k\ge k_0$ there exists a polynomial $P_k\in\mathcal P_k(\mathbb C^{n+m})$ such that $P_{k f}:=P_k(\cdot,f(\cdot))\not\equiv 0$ whose Taylor series at $0\in\mathbb C^n$ has a form 
\[
P_{k f}(z)=\sum_{|\alpha|=s_k+1}^\infty  [P_{k f}]_\alpha\,z^\alpha,\qquad z\in\mathbb C^n.
\]
\end{Lm}
\begin{proof}
Since $d_{s_k,n}<d_{k,n+1}$, the linear map $\pi:\mathcal P_{k}(\mathbb C^{n+1})\rightarrow\mathcal P_{s_k}(\mathbb C^n)$, 
\[
\pi(p)(z):=\sum_{|\alpha|=0}^{s_k}[p_{f_1}]_\alpha z^\alpha,\qquad p_{f_1}:=p(\cdot,f_1(\cdot)),\quad z\in\mathbb C^n,
\] 
has a nonzero kernel. Then as $P_k$ we choose the pullback of a nonzero element of ${\rm ker}\,\pi$ to $\mathbb C^{n+m}$ with respect to the natural projection $\mathbb C^{n+m}\rightarrow\mathbb C^{n+1}$ onto the first $n+1$ coordinates. Since $f_1$ is nonpolynomial, $P_{kf}\not\equiv 0$.
\end{proof}

\noindent (!) In what follows by $\mathcal L_n$  we denote the family of complex lines $l\subset\mathbb C^n$  passing through the origin.\smallskip

For $r>0$, let $l_r \in\mathcal L_n$ be a complex line  such that
\[
\sup_{l_r\cap \mathbb B_r^n}|P_{kf}|=M_{P_{kf}}(r).
\]
Let us identify $l_r$ with $\mathbb C$. Then the univariate entire function $h_k:=P_{kf}|_{l_r}$ has zero of order at least $s_k+1$ at $0$. Let $n_{h_k}(r)$ denote the number of zeros of $h_k$ in $l_r\cap \mathbb B_r^n=\mathbb D_r$ counted with their multiplicities. Then due to the Jensen type inequality, see \cite[Lm.\,1]{VP}, and the Bernstein type inequality of exponent $\mu$ for $\Gamma_f$ (cf. \eqref{e1}),
\[
s_{k}+1\le n_{h_k}(r)\le \frac{m_{h_k}(er)-m_{h_k}(r)}{\ln\left(\frac{1+e^2}{2e}\right)}\le \frac 52\ln\left(\frac{M_{P_{kf}}(er)}{M_{P_{kf}}(r)}\right)\le \frac 52 \mu(k)\ln C(r).
\]
Choosing here $r=1$ we obtain, cf. \eqref{eq2.10}, that there exists $c>0$ such that for all $k\ge k_0$
\[
\mu(k)\ge c k^{1+\frac 1n}.
\]
This implies the required statement:
\[
\varliminf_{k\rightarrow\infty}\frac{\mu(k)}{k^{1+\frac 1n}}\ne 0.
\]
(d) Let $C_i(r)$, $r>0$, be the constant in the Bernstein type inequality of exponent $\mu_i$ for $\Gamma_f$, $i=1,2$.
Then for all $p\in\mathcal P_k(\mathbb C^{n+m})\setminus\{0\}$, $r>0$ and all $\mu\ge\min\{\mu_1,\mu_2\}$ we have
\[
\begin{array}{l}
\displaystyle
\frac{\sup_{\mathbb B^n_{er}} |p_f|}{\sup_{\mathbb B^n_r} |p_f|}\le \min_{i=1,2}\left\{C_i(r)^{\mu_i({\rm deg}\,p)}\right\} \le\min_{i=1,2}\left\{\left(\max\bigl(C_1(r),C_2(r)\bigr)\right)^{\mu_i({\rm deg}(p))}\right\} \medskip\\
\displaystyle \quad\qquad\qquad =\bigl(\max\bigl(C_1(r),C_2(r)\bigr)\bigr)^{\min_{i=1,2}\left\{\mu_i({\rm deg}\,p)\right\}}\le \bigl(\max\bigl(C_1(r),C_2(r)\bigr)\bigr)^{\mu({\rm deg}\,p)}.
\end{array}
\]
This gives the required statement.\smallskip

(e)  For $w\ne 0$ we set $d_w:=\|w\|$.  Then for each $r\ge d_w+1$ the open ball $\mathbb B_r^n(w)$ contains the ball $\mathbb B_{r-d_w}^n$ and the open ball $\mathbb B_{er}^n(w)$ is contained in the ball $\mathbb B_{er+d_w}^n$. Hence,
for all $p\in\mathcal P_k(\mathbb C^{n+m})\setminus\{0\}$ and all $r\ge d_w+1$ we obtain, for $s:=\left\lfloor\ln\left(\frac{er+d_w}{r-d_w}\right)\right\rfloor+1$,
\begin{equation}\label{eq2.11}
\begin{array}{l}
\displaystyle \frac{\sup_{\mathbb B^n_{er}} |p_{f_w}|}{\sup_{\mathbb B^n_r} |p_{f_w}|}:=
\frac{\sup_{\mathbb B^n_{er}(w)} |p_f|}{\sup_{\mathbb B^n_r(w)} |p_f|} \le \frac{\sup_{\mathbb B^n_{er+d_w}} |p_f|}{\sup_{\mathbb B^n_{r-d_w}} |p_f|}\medskip\\
\displaystyle \qquad\qquad\qquad\le\left(\prod_{j=1}^{s-1}C\bigl(e^j(r-d_w)\bigr)\right)^{\mu({\rm deg}\,p)}=:C(r,w)^{\mu({\rm deg}\,p)}.
\end{array}
\end{equation}
From here, using the Hadamard three circle theorem, for all $r\in (0,d_w+1)$ we get
\begin{equation}\label{eq2.12}
\frac{\sup_{\mathbb B^n_{er}} |p_{f_w}|}{\sup_{\mathbb B^n_r} |p_{f_w}|}:=\frac{\sup_{\mathbb B^n_{er}(w)} |p_f|}{\sup_{\mathbb B^n_r(w)} |p_f|}\le C(d_w+1,w)^{\mu({\rm deg}(p))}.
\end{equation}
Inequalities \eqref{eq2.11} and \eqref{eq2.12} show that $f_w$ admits the Bernstein type inequality of exponent $\mu$ as well.\smallskip

(f) By definition, 
\[
\Gamma_{f_1\times f_2}=\{ (z,f_1(z),w, f_2(w))\, : \,  z\in\mathbb C^{n_1},\ w\in\mathbb C^{n_2}\}\subset\mathbb C^{n_1+n_2+m_1+m_2}.
\]
For each $r>0$  and $p\in\mathcal P(\mathbb C^{n_1+n_2+m_1+m_2})$  applying Bernstein type inequalities of exponents $\mu_1$ and $\mu_2$ to restrictions of $p$ to cross sections $\Gamma_{f_1}\times \{(w,f_2(w))\}$ and $\{(z,f_1(z))\}\times\Gamma_{f_2}$, for fixed $z\in\mathbb C^{n_1}$, $w\in\mathbb C^{n_2}$, we get
\[
\begin{array}{l}
\displaystyle
\sup_{\mathbb B_{er}^{n_1}\times \mathbb B_{er}^{n_2}}|p_{f_1\times f_2}|\le C_1(r)^{\mu_1({\rm deg}\,p)}\sup_{\mathbb B_{r}^{n_1}\times \mathbb B_{er}^{n_2}}|p_{f_1\times f_2}|\medskip\\
\displaystyle \quad\qquad\qquad\qquad\le C_1(r)^{\mu_1({\rm deg}\,p)} C_2(r)^{\mu_2({\rm deg}\,p)}\sup_{\mathbb B_{r}^{n_1}\times \mathbb B_{r}^{n_2}}|p_{f_1\times f_2}|\medskip\\
\displaystyle\quad\qquad\qquad\qquad\le \bigl(C_1(r)\, C_2(r)\bigr)^{\max_{i=1,2}\{\mu_i({\rm deg}\,p)\}}\sup_{\mathbb B_{r}^{n_1}\times \mathbb B_{r}^{n_2}}|p_{f_1\times f_2}|.
\end{array}
\]
Replacing products of balls by suitable inscribed and circumscribed balls of $\mathbb C^{n_1+n_2+m_1+m_2}$ and arguing as in the proof of (e) we obtain that
$\Gamma_{f_1\times f_2}$ admits the Bernstein type inequality of exponent $\max(\mu_1,\mu_2)$.

Now, assume that $\Gamma_{f_1\times f_2}$ admits the Bernstein type inequality of exponent $\mu$. 
Applying this inequality to polynomials $p$ pulled back from $\mathbb C^{n_i+m_i}$ by means of the natural projections $\mathbb C^{n_1+n_2+m_1+m_2}\rightarrow \mathbb C^{n_i+m_i}$, $i=1,2$, we obtain
\[
\sup_{\mathbb B_{er}^{n_i}}|p_{f_i}|=\sup_{\mathbb B_{er}^{n_1+n_2}}|p_{f_1\times f_2}|\le C(r)^{\mu({\rm deg}\,p)}
\sup_{\mathbb B_{r}^{n_1+n_2}}|p_{f_1\times f_2}|=C(r)^{\mu({\rm deg}\,p)}\,
\sup_{\mathbb B_{r}^{n_i}}|p_{f_i}|.
\]
Thus, $\Gamma_{f_i}$, $i=1,2$, admit the Bernstein type inequality of exponent $\mu$.

The proof of the theorem is complete.

\section{Proof of Theorem \ref{theo1.3}}
\begin{proof}
(a) Approximating polynomial $p=1$ by the sequence of polynomials of degree $k$ $\{p_i\}_{i\in\mathbb N}$, 
\begin{equation}\label{equ3.21}
p_i(z):=\frac{i+z_1^k}{\sup_{z\in K} (i+z_1^k)},\quad z=(z_1,\dots, z_{n+m})\in\mathbb C^{n+m},
\end{equation}
we conclude that $u_{K,\mu}^k\ge 0$.
Then for $z_1,z_2\in\mathbb C^n$ we have 
\[
\begin{array}{l}
\displaystyle
u_{K,\mu}^k(z_1;f)-u_{K,\mu}^k(z_2;f)\medskip\\
\displaystyle
\le \sup\left\{\frac{\ln^+ |p_f(z_1)|-\ln^+ |p_f(z_2)|}{\max\bigl(1,\mu(k)\bigr)}\, :\, p\in\mathcal P(\mathbb C^{n+m}),\ {\rm deg}\,p= k,\ \sup_{K}|
p_f|=1\right\}\medskip\\
\displaystyle \le \sup\left\{\frac{|p_f(z_1)-p_f(z_2)|}{\max\bigl(1,\mu(k)\bigr)}\, :\, p\in\mathcal P(\mathbb C^{n+m}),\ {\rm deg}\,p= k,\ \sup_{K}|p_f|=1\right\}\medskip\\
\displaystyle \le  \sup\left\{\frac{C(z_1,z_2)\,\|z_1-z_2\|}{\max\bigl(1,\mu(k)\bigr)}\, :\, p\in\mathcal P(\mathbb C^{n+m}),\ {\rm deg}\,p= k,\ \sup_{K}|p_f|=1\right\}\medskip\\
\displaystyle \le C(z_1,z_2)\,\|z_1-z_2\|.
\end{array}
\]
Here we use that $\ln^+ x:=\max(0,\ln x)$, $x>0$, is a Lipschitz function with Lipschitz constant $1$ and  the uniform boundedness of the family 
\[
\{p_f\, :\, p\in\mathcal P(\mathbb C^{n+m}),\ {\rm deg}\,p= k,\ \sup_{K}|p_f|=1\}|_{U}\subset C(U)
\]
on each compact subset $U\subset\mathbb C^n$. The constant $C$ in the above inequality is obtained by applying the Cauchy estimates for derivatives of $p_f$ on an open polydisk containing $z_1$ and $z_2$.

The above inequality shows that the function $u_{K,\mu}^k$ is locally Lipschitz and, in particular, it is continuous. Then, by definition, it is plurisubharmonic.\smallskip

(b) Clearly, if $\Gamma_f$ admits the Bernstein type inequality of exponent $\mu$, then the function $u_{K,\mu}$ is locally bounded from above. Conversely, assume that the function $u_{K,\mu}$ is locally bounded from above. Then according to the Hartogs lemma on subharmonic functions, the sequence of continuous plurisubharmonic functions $\{u_{K,\mu}^k\}_{k\in\mathbb N}$ is uniformly bounded from above on each compact subset of $\mathbb C^n$. This implies fulfillment of inequality \eqref{e2} and so due to Theorem \ref{theo1}\,(b), $\Gamma_f$ admits the Bernstein type inequality of exponent $\mu$. \smallskip

(c) Suppose that $u_{K,\mu,\bar k}\not\equiv 0$ for every subsequence $\bar k\subset\mathbb N$ but $\mu$ is not optimal. Then there exists a function $\mu_1:\mathbb Z_+\rightarrow\mathbb R_+$ such that $\mu_1\le c\mu$ for some $c>0$, 
\[
\varlimsup_{k\rightarrow\infty}\frac{\mu(k)}{\mu_1(k)}=\infty
\]
and $\Gamma_f$ admits the Bernstein type inequality of exponent $\mu_1$. Let $\bar k=\{k_i\}_{i\in\mathbb N}\subset\mathbb N$ be a subsequence such that
\[
\lim_{i\rightarrow\infty}\frac{\mu(k_i)}{\mu_1(k_i)}=\infty.
\]
We have, cf. Theorem \ref{theo1}\,(b),
\[
\begin{array}{l}
\displaystyle
0\le u_{K,\mu;\bar k}(z;f):=\varlimsup_{i\rightarrow\infty}\sup\left\{\frac{\ln |p_f(z)|}{\max\bigl(1,\mu(k_i)\bigr)}\, :\, p\in\mathcal P(\mathbb C^{n+m}),\ {\rm deg}\,p= k_i,\ \sup_K |p_f|=1\right\}\medskip\\
\displaystyle =\varlimsup_{i\rightarrow\infty}\sup\left\{\frac{\max\bigl(1,\mu_1(k_i)\bigr)}{\max\bigl(1,\mu(k_i)\bigr)}\cdot\frac{\ln |p_f(z)|}{\max\bigl(1,\mu_1(k_i)\bigr)}\, :\, p\in\mathcal P(\mathbb C^{n+m}),\ {\rm deg}\,p= k_i,\ \sup_K |p_f|=1\right\}\\
\\
\displaystyle \le \varlimsup_{i\rightarrow\infty}\frac{\max\bigl(1,\mu_1(k_i)\bigr)}{\max\bigl(1,\mu(k_i)\bigr)} \ln C(K;\|z\|)=0.
\end{array}
\]
Here $C(K;r)$, $r>0$, is the constant in \eqref{e2} for the exponent $\mu_1$.

This implies that $u_{K,\mu;\bar k}= 0$, a contradiction showing that $\mu$ is optimal.\smallskip

Conversely, suppose that $\Gamma_f$ admits the Bernstein type inequality of an optimal exponent $\mu$ but there exists a subsequence $\bar k=\{k_i\}_{i\in\mathbb N}\subset\mathbb N$ such that $u_{K,\mu;\bar k}= 0$. Then the Hartogs lemma on subharmonic functions implies that for each $\ell\in\mathbb N$ there exists a number $i(\ell)\in\mathbb N$ such that for all $i\ge i(\ell)$
\begin{equation}\label{eq3.11}
\sup_{\mathbb B_\ell^n} u_{K,\mu}^{k_i}\le \frac{1}{\ell}.
\end{equation}
Passing to a subsequence, if necessary, we may assume that $\{i(\ell)\}_{\ell\in\mathbb N}$ is an increasing sequence.
We set $\bar k_*:=\{k_{i(\ell)}\}_{\ell\in\mathbb N}$.
Let us define a function $\mu_1:\mathbb Z_+\rightarrow\mathbb R_+$ by the formula
\[
\mu_1(k)=\left\{
\begin{array}{cccc}
\mu(k)&{\rm if}&k\not\in \bar k_*\medskip\\
\frac{\mu(k)}{\ell}&{\rm if}&k=k_{i(\ell)}\ {\rm for\ some}\ \ell\in\mathbb N.
\end{array}
\right.
\]
Then due to \eqref{eq3.11} we obtain
\[
\begin{array}{r}
\displaystyle
u_{K,\mu_1}(z;f)=\varlimsup_{k\rightarrow\infty}u_{K,\mu_1}^k(z;f)=\max\left(\varlimsup_{\stackrel{ k\rightarrow\infty}{_{k\not\in \bar k_*}}}u_{K,\mu}^k(z;f), \varlimsup_{\ell\rightarrow\infty} \ell\, u_{K,\mu}^{i(\ell)}(z;f) \right)\medskip\\
\displaystyle \le\max\left(u_{K,\mu}(z;f), 1 \right).\qquad\qquad\qquad\qquad\quad\,
\end{array}
\]
Thus, $u_{K,\mu_1}$ is locally bounded from above and so part (b) of the theorem implies that $\Gamma_f$ admits the Bernstein type inequality of exponent $\mu_1$. Clearly, $\mu_1\le\mu$. Thus, due to the optimality of $\mu$, function $\mu_1$ must be equivalent to $\mu$. However, this is wrong as $\varlimsup_{k\rightarrow\infty}\frac{\mu(k)}{\mu_1(k)}=\infty$. This contradiction shows that $u_{K,\mu;\bar k}\not\equiv 0$ for every subsequence $\bar k\subset\mathbb N$.\smallskip 

(d) If $\Gamma_f$ admits the Bernstein type inequality of exponent $\mu$, then due to  Theorem \ref{theo1}\,(b), $u_K^k(r;f)\le\ln C(K;r)\mu(k)$ for all $k\in\mathbb Z_+$, $r>0$. This implies that $\langle u_K^{{\bf\cdot}}(r;f)\rangle\le\langle\mu\rangle$ for all $r>0$. Assume, in addition, that $\langle\mu\rangle\in \mathscr U_f^K$.
Then $\langle\mu\rangle$ is the maximal element of $\mathscr U_f^K$.
If $\mu$ is not optimal for $\Gamma_f$, then there is an exponent $\mu_1$ for $\Gamma_f$ such that $\langle\mu_1\rangle<\langle\mu\rangle$. Since $\langle\mu\rangle\in\mathscr U_f^K$, we must have $\langle\mu\rangle\le\langle\mu_1\rangle$, a contradiction showing that $\mu$ is optimal.

Conversely, suppose that $\mu$ is an optimal exponent for $\Gamma_f$. We require
\begin{Lm}\label{lemma4.1}
There exist $r>0$ and $c>0$ such that for all $k\in\mathbb N$
\[
u_K^k(r;f)\ge c\max\bigl(1,\mu(k)\bigr).
\]
\end{Lm}
\begin{proof}
Assume, on the contrary, that for each $r>0$ and $c>0$ there exists an integer $k_{r,c}\in\mathbb N$ such that
\begin{equation}\label{eq4.1}
u_K^{k_{r,c}}(r;f)< c\max\bigl(1,\mu(k_{r,c})\bigr).
\end{equation}
Choose $r:=j$, $c:=\frac 1j$ and set $k_{j}:=k_{j,1/j}$, $j\in\mathbb N$. Let us show that $\varlimsup_{j\rightarrow\infty} k_j=\infty$. Indeed, for otherwise, the sequence $\bar k:=\{k_j\}_{j\in\mathbb N}$ is bounded. In particular, there exists an element $k'\ge 1$ of $\bar k$ such that (as each $u_{K,\mu}^k$ is plurisubharmonic) $u_{K,\mu}^{k'}= 0$, a contradiction.

Since each $u_{K,\mu}^{k}$ is plurisubharmonic, inequality \eqref{eq4.1} implies that the function
\[
u_{K,\mu;\bar k}:=\varlimsup_{j\rightarrow\infty}u_{K,\mu}^{k_j}
\]
is identically zero. Due to part (c) of the theorem, this contradicts the optimality of $\mu$.

The proof of the lemma is complete.
\end{proof}
As the corollary of the lemma we get
\[
(\langle\mu\rangle\ge)\,\langle u_K^{\cdot}(r;f)\rangle\ge \langle\mu\rangle.
\]
Thus $\langle\mu\rangle=\langle u_K^{\cdot}(r;f)\rangle\in\mathscr U_f^K$ is maximal. 

The proof of the theorem is complete.
\end{proof}
\section{Proof of Theorem \ref{theo1.4}}
In the proof of the theorem we use the following Bernstein type inequality for exponential polynomials established in \cite{VP}.

Let
\[
g(z)=\sum_{j=1}^{n}p_j(z) e^{q_j z},\qquad z\in\mathbb C,
\]
where $p_j\in\mathcal P(\mathbb C)$,  ${\rm deg}\, p_j=d_j$ and $q_j\in \mathbb C$ are pairwise disjoint, $1\le j\le n$, be an exponential polynomial on $\mathbb C$. The expression
\[
m(g):=\sum_{j=1}^n (1+d_j)
\]
is called the {\em degree} of $g$. In turn, the exponential type of $g$ is defined by the
formula
\[
\epsilon(g):=\max_{1\le j\le n}|q_j|.
\]
Then \cite[p.\,27, Eq.\,(21)]{VP} asserts that for each $r>0$,
\begin{equation}\label{eq4.18}
\sup_{\mathbb D_{er}}|g|\le e^{m(g)+2er\epsilon(g)}\sup_{\mathbb D_r}|g|.
\end{equation}

\begin{proof}[Proof of Theorem \ref{theo1.4}]
First, we prove the theorem for $l=1$ and $P,Q$ the identity automorphisms. In this case, 
\[
f_{P,Q}(z)=f(z)=\bigl(e^{\alpha_1 z},\dots, e^{\alpha_m z}\bigr),\quad z\in\mathbb C,
\]
where $\alpha_1,\dots, \alpha_m$ are linearly independent over $\mathbb Q$ complex numbers.
Let $p\in\mathcal P_k(\mathbb C^{m+1})$,
\[
p(z,w):=\sum_{ |\gamma|\le k}c_\gamma\, z^{\gamma_1} w_1^{\gamma_2}\cdots w_m^{\gamma_{m+1}},\quad z\in\mathbb C,\quad w=(w_1,\dots, w_m)\in\mathbb C^m
\]
(here $\gamma=(\gamma_1,\dots,\gamma_{m+1})\in\mathbb Z_+^{m+1}$ and all $c_\gamma\in\mathbb C$).

Since $\alpha_1,\dots ,\alpha_m$ are linearly independent over $\mathbb Q$,
\[
p_f(z):=p(z, f(z))=\sum_{ |\gamma|\le k}c_\gamma\, z^{\gamma_1}e^{(\gamma_2 \alpha_1+\cdots +\gamma_{m+1} \alpha_m) z}=\sum_{|\gamma'|\le k} p_{\gamma'}(z) e^{(\gamma_2 \alpha_1+\cdots +\gamma_{m+1} \alpha_m) z};
\]
here $\gamma':=(\gamma_2,\dots,\gamma_{m+1})$ and
$p_{\gamma'}\in\mathcal P_{k-|\gamma'|}(\mathbb C)$. 

Then the exponential type of $p_f$ is
\[
\epsilon(p_f):=\max_{|\gamma|\le k} \left|\sum_{j=2}^{m+1} \gamma_j \alpha_{j-1}\right|\le k\max_{1\le j\le m}|\alpha_j|=:k \bar \alpha
\]
and the degree  of $p_f$ satisfies the inequality
\[
m(p_f)\le \sum_{|\gamma'|\le k}(1+(k-|\gamma'|))= \sum_{|\gamma|\le k} 1=d_{k,m+1}=:\binom{m+1+k}{k}.
\]
Hence, in this case \eqref{eq4.18} yields the inequality
\begin{equation}\label{eq4.19}
\sup_{\mathbb D_{er}}|p_f|\le e^{2er k\bar \alpha+d_{k,m+1}}\sup_{\mathbb D_{r}}|p_f|,\qquad r>0.
\end{equation}
Note that for all $k\ge 1$
\[
\bigl(e^{2er k\bar \alpha+d_{k,m+1}}\bigr)^{\frac{1}{k^{m+1}}}< 5^{m+1} e^{2er \bar\alpha}=:C(r,f).
\]
This and \eqref{eq4.19} show that $\Gamma_f$ satisfies the Bernstein type inequality of exponent $\mu_{{\rm id}}^{m+1}$. 

Next, we show that this exponent is optimal.

Since complex numbers $\alpha_1,\dots,\alpha_m$ are linearly independent over $\mathbb Q$, the restriction maps $\mathcal P_k(\mathbb C^{m+1})\rightarrow \mathcal P_k(\mathbb C^{m+1})|_{\Gamma_f}$, $k\in\mathbb Z_+$, are linear isomorphisms. Thus 
arguing as in the proof of Theorem \ref{theo1}\,(c) we conclude that for each $k\ge 1$ there exists a polynomial $g_k\in\mathcal P_k(\mathbb C^{m+1})$ such that  $g_{kf}\not\equiv 0$ and has zero of multiplicity $d_{k,m+1}-1$ at $0\in\mathbb C$.
Then due to the Jensen type inequality, see \cite[Lm.\,1]{VP}, and the Bernstein type inequality of exponent $\mu$ for $\Gamma_f$ (cf. the proof of Theorem \ref{theo1}\,(c) for similar arguments),
\[
d_{k,m+1}-1\le  \frac{m_{g_k}(er)-m_{g_k}(r)}{\ln\left(\frac{1+e^2}{2e}\right)}\le  \frac 52 \mu(k)\ln C(r).
\]
Taking here $r=1$ we obtain that there exists $c>0$ such that for all $k\ge 1$
\[
\mu(k)\ge c k^{m+1}:=c\mu_{{\rm id}}^{m+1}(k).
\]
This and \eqref{eq4.19} show that $\mu_{{\rm id}}^{m+1}$ is the optimal exponent for the Bernstein type inequality on $\Gamma_f$ completing the proof of the theorem in this particular case.\smallskip

We deduce the general case from the one just proved by means of the following result.
\begin{Lm}\label{lem4.1}
Suppose $\Gamma_f\subset\mathbb C^{n+m}$, $f:\mathbb C^n\rightarrow\mathbb C^m$, admits the Bernstein type inequality of exponent $\mu_{{\rm id}}^r$, $r\ge 1$. Let $P$ and $Q$ be holomorphic polynomial automorphisms of $\mathbb C^n$ and $\mathbb C^m$, respectively. Then $f_{P,Q}:=Q\circ f\circ P$ admits the Bernstein type inequality of exponent $\mu_{{\rm id}}^r$ as well. Moreover, if $\mu_{{\rm id}}^r$ is optimal for $\Gamma_f$, then it is optimal for $\Gamma_{f_{P,Q}}$ as well.
\end{Lm}
\begin{proof}
By definition, there are some $s, t\in\mathbb N$ such that the coordinates of maps $P^{\pm 1}$ and $Q^{\pm 1}$ are holomorphic polynomials in $\mathcal P_{s}(\mathbb C^n)$ and $\mathcal P_t(\mathbb C^m)$, respectively. Then the correspondence $h(z,w)\mapsto h(P^{-1}(z),Q(w))$, $z\in\mathbb C^n$, $w\in\mathbb C^m$, determines a linear injective map $I:\mathcal P_k(\mathbb C^{n+m})\rightarrow \mathcal P_{k\max(s,t)}(\mathbb C^{n+m})$.
 By $K$ we denote the image of the closure of $\mathbb B^n$ under map $P$, i.e. $K:=P\bigl({\rm cl}(\mathbb B^n)\bigr)$.
Since $P$ is a holomorphic automorphism of $\mathbb C^n$, $K$ is a nonpluripolar compact subset of $\mathbb C^n$.
By definition, cf. Theorem \ref{theo1.3},
\[
\begin{array}{l}
\displaystyle
u^{k}_{{\rm cl}(\mathbb B^n),\,\mu_{{\rm id}}^r}(z;f_{P,Q})=
\sup\left\{\frac{\ln |h_{f_{P,Q}}(z)|}{\max\bigl(1,k^r\bigr)}\, :\, h\in\mathcal P(\mathbb C^{n+m}),\ {\rm deg}\,h= k,\ \sup_{\mathbb B^n}|h_{f_{P,Q}}|=1\right\}\\
\\
\displaystyle =\sup\left\{\frac{\ln |\bigl(I(h)\bigr)_{f}(P(z))|}{\max\bigl(1,k^r\bigr)}\, :\, h\in\mathcal P(\mathbb C^{n+m}),\ {\rm deg}\,h= k,\ \sup_{K}|\bigl(I(h)\bigr)_{f}|=1\right\}\\
\\
\displaystyle \le \sup\left\{\frac{\ln |g_{f}(P(z))|}{\max\bigl(1,k^r\bigr)}\, :\, g\in\mathcal P(\mathbb C^{n+m}),\ {\rm deg}\,g= k\max(s,t),\ \sup_{K}|g_{f}|=1\right\}\\
\\
\displaystyle \le\bigl(\max(s,t)\bigr)^r u_{K,\,\mu_{{\rm id}}^r}^{k\max(s,t)}(P(z);f),\qquad z\in\mathbb C^n.
\end{array}
\]
This yields\smallskip
\begin{equation}\label{eq4.20}
\begin{array}{r}
\displaystyle
u_{{\rm cl}(\mathbb B^n),\,\mu_{{\rm id}}^r}(z;f_{P,Q}):=\varlimsup_{k\rightarrow\infty}u^{k}_{{\rm cl}(\mathbb B^n),\,\mu_{{\rm id}}^r}(z; f_{P,Q}) \le \bigl(\max(s,t)\bigr)^r\varlimsup_{k\rightarrow\infty}u_{K,\,\mu_{{\rm id}}^r}^{k\max(s,t)}(P(z);f)\quad\medskip\\
\displaystyle \le \bigl(\max(s,t)\bigr)^r u_{K,\,\mu_{{\rm id}}^r} (P(z);f),\quad z\in\mathbb C^n.
\end{array}
\end{equation}
Since the function  $u_{K,\,\mu_{{\rm id}}^r}(\cdot\, ; f)$ is locally bounded from above by Theorem \ref{theo1.3}\,(b),  the latter inequality implies that the function $u_{{\rm cl}(\mathbb B^n),\,\mu_{{\rm id}}^r}(\cdot\, ;f_{P,Q})$ is locally bounded from above as well. Then by Theorem \ref{theo1.3}\,(b) graph $\Gamma_{f_{P,Q}}$ admits the Bernstein type inequality of exponent $\mu_{{\rm id}}^r$.

Further, suppose that $\mu_{{\rm id}}^r$ is optimal for $\Gamma_f$. Let $\bar k=\{k_j\}_{j\in\mathbb N}\subset\mathbb N$ be a subsequence. Applying the arguments similar to the above one to functions
$f_{P,Q}=Q\circ f\circ P$ and automorphisms $P^{-1}, Q^{-1}$ instead of $f$ and $P, Q$ as in the hypothesis of the lemma, we get  for $\bar{k}_{s,t}:=\{k_j \max(s,t)\}_{j\in\mathbb N}$ (cf. \eqref{eq4.20})
\begin{equation}\label{eq4.21}
\begin{array}{l}
\displaystyle
u_{K,\,\mu_{{\rm id}}^r;\bar k}(z;f):=\varlimsup_{j\rightarrow\infty}u^{k_j}_{K,\,\mu_{{\rm id}}^r}(z; f) \le \bigl(\max(s,t)\bigr)^r\varlimsup_{j\rightarrow\infty}u_{{\rm cl}(\mathbb B^n),\,\mu_{{\rm id}}^r}^{k_j\max(s,t)}(P^{-1}(z);f_{P,Q})\medskip\\
\displaystyle \qquad\qquad\qquad\ = \bigl(\max(s,t)\bigr)^r u_{{\rm cl}(\mathbb B^n),\,\mu_{{\rm id}}^r;\bar {k}_{s,t}} (P^{-1}(z);f_{P,Q}),\quad  z\in\mathbb C^n.
\end{array}
\end{equation}
Since $\mu_{{\rm id}}^r$ is optimal for $\Gamma_f$, Theorem \ref{theo1.3}\,(c) implies that $u_{K,\,\mu_{{\rm id}}^r;\bar k}(\cdot\,; f)\not\equiv 0$. Hence, equation \eqref{eq4.21} shows that $u_{{\rm cl}(\mathbb B^n),\,\mu_{{\rm id}}^r;\bar {k}_{s,t}} (\cdot\, ;f_{P,Q})\not\equiv 0$ as well (recall that all functions in \eqref{eq4.21} are nonnegative). 

Let us check a similar statement for an arbitrary sequence $\bar n=\{n_j\}_{j\in\mathbb N}\subset\mathbb N$. 

We set 
\[
\bar k=\{k_j\}_{j\in\mathbb N},\quad k_j:=\left\lfloor\frac{n_j}{\max(s,t)}\right\rfloor,\quad j\in\mathbb N.
\]
Then, by the definition of $u_{{\rm cl}(\mathbb B^n),\,\mu_{{\rm id}}^r}^k(\cdot\, ; f_{P,Q})$,
\[
u_{{\rm cl}(\mathbb B^n),\,\mu_{{\rm id}}^r}^{k_j\max(s,t)}(z ; f_{P,Q})\le \frac{\max(1,n_j^r)}{\max(1, (k_j\max(s,t))^r)} u_{{\rm cl}(\mathbb B^n),\,\mu_{{\rm id}}^r}^{n_j}(z ; f_{P,Q})\le 2^r u_{{\rm cl}(\mathbb B^n),\,\mu_{{\rm id}}^r}^{n_j}(z ; f_{P,Q}).
\]
This yields
\begin{equation}\label{eq4.22}
u_{{\rm cl}(\mathbb B^n),\,\mu_{{\rm id}}^r;\bar {k}_{s,t}} (z ;f_{P,Q})\le 2^r u_{{\rm cl}(\mathbb B^n),\,\mu_{{\rm id}}^r;\bar {n}} (z ;f_{P,Q}).
\end{equation}
Since  $u_{{\rm cl}(\mathbb B^n),\,\mu_{{\rm id}}^r;\bar {k}_{s,t}} (\cdot\, ;f_{P,Q})\not\equiv 0$, the latter shows that  $u_{{\rm cl}(\mathbb B^n),\,\mu_{{\rm id}}^r;\bar {n}} (\cdot\, ;f_{P,Q})\not\equiv 0$ as well. Thus, Theorem \ref{theo1.3}\,(c) implies that the exponent $\mu_{{\rm id}}^r$ is optimal for $\Gamma_{f_{P,Q}}$. 

The proof of the lemma is complete.
\end{proof}
Now, if $f_j:\mathbb C\rightarrow\mathbb C^{m_j}$, $1\le j\le l$, are exponential maps of maximal transcendence degrees, then due to Theorem \ref{theo1}\,(f) and the above considered case of $l=1$ and $P,Q$ the identity automorphisms, the graph of $f_1\times\cdots\times f_l:\mathbb C^l\rightarrow\mathbb C^{m_1+\cdots +m_l}$ satisfies the Bernstein type inequality of exponent $\mu_{{\rm id}}^{\bar m +1}$, $\bar m=\max_{1\le j\le l} m_j$. Hence, by Lemma \ref{lem4.1} graph $\Gamma_{F_{P,Q}}$, $F_{P,Q}:=Q\circ (f_1\times\cdots\times f_l)\circ P$, satisfies the Bernstein type inequality of exponent $\mu_{{\rm id}}^{\bar m +1}$ as well.
Since, by Theorem \ref{theo1}\,(f), $\mu_{{\rm id}}^{\bar m+1}$ is optimal for $f_1\times\cdots\times f_l$, Lemma \ref{lem4.1} implies that $\mu_{{\rm id}}^{\bar m +1}$ is optimal for $\Gamma_{F_{P,Q}}$.

This completes the proof of the theorem.
\end{proof}
\section{Proof of Theorem \ref{theo1.5}}
According to \cite[Th.\,2.5\,(c)]{B} there exist increasing sequences $\{n_j\}_{j\in\mathbb N}\subset\mathbb N$, $\{r_j\}_{j\in\mathbb N}\subset\mathbb R_+$ converging to $\infty$ and a nonincreasing sequence $\{\epsilon_j\}\subset\mathbb R_+$ converging to $0$ such that for all $g\in\mathcal P_{n_j}(\mathbb C^{n+1})$ and all $0\le r\le r_j$,
\begin{equation}\label{eq5.24}
M_{g_f}(er)\le e^{C_{\rho_f} n_j^{2+\epsilon_j}} M_{g_f}(r),
\end{equation} 
for some constant $C_{\rho_f}>0$ depending on the order of  $f$ only. 

Without loss of generality we assume that $r_1=0$ and $n_1=1$

For $j\in\mathbb N$, $r\in [r_j, r_{j+1})$ we set
\[
C_j(r):=\sup\left\{\frac{M_{g_f}(er)}{M_{g_f}(r)}\, :\, g\in \mathcal P_{n_{j}}(\mathbb C^{n+1})\setminus\{0\}\right\}.
\]
Since $f$ is nonpolynomial and the space $\mathcal P_{n_{j-1}}(\mathbb C^{n+1})$ is finite dimensional, each $C_j(r)<\infty$.  
We define 
\begin{equation}\label{eq5.25}
C(r):=\max\bigl(C_j(r),e^{C_{\rho_f}}\bigr)\quad {\rm for}\quad r_j\le r< r_{j+1},\quad j\in\mathbb N.
\end{equation}
\begin{Lm}\label{lem5.1}
For all $g\in\mathcal P_{n_j}(\mathbb C^{n+1})$, $j\in\mathbb N$, and all $r>0$ 
\begin{equation}\label{eq5.26}
M_{g_f}(er)\le C(r)^{\, n_j^{2+\epsilon_j}}M_{g_f}(r).
\end{equation}
\end{Lm}
\begin{proof}
We consider two cases.

(1) If $0\le r< r_{j}$, then \eqref{eq5.26} follows from \eqref{eq5.24}.\smallskip

(2) If $r_{k}\le r< r_{k+1}$ for some $k\ge j$, then 
\[
M_{g_f}(er)\le C_k(r) M_{g_f}(r)\le C(r)^{\, n_j^{2+\epsilon_j}} M_{g_f}(r)
\]
by the definition of $C(r)$.

The proof of the lemma is complete.
\end{proof}

Now, let $\nu$ be an exponent in the Bernstein type inequality on $\Gamma_f$ (existing by Theorem \ref{theo1}\,(a)). We define a function $\mu:\mathbb Z_+\rightarrow\mathbb R_+$ by the formula
\[
\mu(k)=\left\{
\begin{array}{ccc}
\nu(k)&{\rm if}&k\not\in\{n_j\}_{j\in\mathbb N}\medskip\\
n_j^{2+\epsilon_j}&{\rm if}&k=n_j,\, j\in\mathbb N.
\end{array}
\right.
\]
Lemma \ref{lem5.1} and the definition of the exponent $\nu$ easily imply that $\Gamma_f$ satisfies the Bernstein type inequality of exponent $\mu$. Note that
\[
\varliminf_{k\rightarrow\infty}\frac{\ln\mu(k)}{\ln k}\le \varliminf_{j\rightarrow\infty}\frac{(2+\epsilon_j)\cdot\ln n_j}{\ln n_j}=2.
\]
This gives the right-hand side inequality of the theorem.
The left-hand side inequality,
\[
1+\frac 1n \le \varliminf_{k\rightarrow\infty}\frac{\ln\mu(k)}{\ln k},
\]
follows from Theorem \ref{theo1}\,(c).

The proof of the theorem is complete.

\section{Proofs of Theorems \ref{theo1.7} and \ref{theo1.8}}
\subsection{Theorem A} In this part we discuss an auxiliary result used in the proofs of the theorems.
For its formulation, we require some definitions. 

In what follows, for each $l\in\mathcal L_n$, the family of complex lines  passing through  $0\in\mathbb C^n$, we naturally identify $l\cap\mathbb B_r^n$ with $\mathbb D_r$.

 Assume that $f:\mathbb D_r\rightarrow\mathbb C$ is holomorphic. 
The {\em Bernstein index} $b_f$ of $f$ is given by the formula
\[
b_f(r):=\sup\{m_f(es)-m_f(s)\},
\]
where the supremum is taken over all $\mathbb D_{es}\Subset\mathbb D_r$. (We assume that $b_f(\cdot)=0$ for $f=0$.) The index is 
finite for all $f$ defined in  neighbourhoods of the closure of $\mathbb D_r$. 

Let $g$ be a holomorphic function in the domain $\mathbb B_{tr}^n\times \mathbb D_{3M}\subset\mathbb C^{n+1}$, $r\in\mathbb R_+$, $t\in [1,9]$, $M\in\mathbb R_+\cup\{\infty\}$ (here $\mathbb D_\infty:=\mathbb C$). For every $l\in\mathcal L_n$ we determine
\[
g_l:=g|_{\Omega_l},\qquad \Omega_l:=(l\cap\mathbb B_{tr}^n)\times \mathbb D_{3M}.
\]

We write $g\in {\mathcal F}_{p,q}(r;t;M)$ for some $p,q\ge 0$ if
\begin{equation}
\begin{array}{l}
\displaystyle
M_{g_l(\cdot,w)}(tr)\le e^p\cdot M_{g_l(\cdot,w)}(r)\quad {\rm for\ all}\quad l\in\mathcal L_n,\ w\in\mathbb D_{3M}\quad {\rm and}\medskip\\
b_{g(z,\cdot)}(3M)\le q\quad {\rm for\ all}\quad z\in\mathbb B_{tr}^n .
\end{array}
\end{equation}
\begin{E}
{\rm One can easily check by means of the classical Bernstein inequality that a holomorphic polynomial of degree $d$ on $\mathbb C^{n+1}$ is in $\mathcal F_{p,q}(r ; t ; M)$ with $p = d \ln t$ and $q = d$. See \cite{B} for other examples.}
\end{E} 

\begin{Teo}[\mbox{cf. \cite[Theorem~2.8]{B}}]\label{teoA}
Assume that $f$ is of class $\mathscr C$. Then there exist numbers $k_0,r_0\ge 1$, a continuous increasing to $\infty$ function $r: [k_0,\infty)\rightarrow [r_0,\infty)$ and a continuous decreasing to $0$ function $\varepsilon : [k_0,\infty)\rightarrow\mathbb R_+$ such that for all $k\ge k_0$, $r(k)\ge r_0$, every $g\in\mathcal F_{p,q}(er(k);e;M_f(e^2r(k)))$ with $p\le k$ and every $0< r\le r(k)$ the following inequalities hold for $g_f:=g(\cdot, f(\cdot))$:
\begin{itemize}
\item[(a)]
\[
\sup_{\mathbb B^n\times\mathbb D}|g|\le e^{Ck^{1+\varepsilon(k)}\ln r(k)\max\{p,q\}}M_{g_f}(1);
\]
\item[(b)]
\[
\frac{M_{g_f}(er)}{M_{g_f}(r)}\le e^{Ck^{1+\varepsilon(k)}\max\{p,q\}}.
\]
\end{itemize}
Here for $\rho_f<\infty$ the constant $C$ depends on the value of the limit superior of condition \eqref{cI} and $\rho_f$, and for $\rho_f=\infty$ the constant $C=1$.

Moreover, 
\begin{itemize}
\item[(1)]
If $\rho_f<\infty$, then $\varepsilon=0$ and  function $r$ is the right inverse of the nondecreasing function $k:[r_0,\infty)\rightarrow\mathbb R_+$,
\begin{equation}\label{equ6.32}
k(r):=\frac{m_f(e^{-\alpha_{\rho_f}}r)- m_f(e^{-2\alpha_{\rho_f}}r)-1}{9(\sqrt e +1)^2(\rho_f^2+1)(17+2\ln(\rho_f+1))},
\end{equation}
where $\alpha_{\rho_f}:=\min\bigl(1,\ln\bigl(1+\frac{1}{\rho_f}\bigr)\bigr)$.
\item[(2)]
If $\rho=\infty$, then $r(k)=\frac{1}{e^2}m_f^{-1}(k^{1+\varepsilon''(k)})\le e^{k^{\delta(k)}}$, $k\ge k_0$, for some continuous functions $\varepsilon'',\delta : [k_0,\infty)\rightarrow\mathbb R_+$ decreasing to $0$ as $k\rightarrow\infty$.
\end{itemize}
\end{Teo}
\begin{proof}
For $\rho_f<\infty$ the statement of the theorem is the direct consequence of Theorem 2.8 of \cite{B}.  The latter is proved under the assumption
\begin{equation}\label{equ6.33}
\varlimsup_{t\rightarrow\infty}\left(\frac{\phi_f(t+\alpha_{\rho_f})-\phi_f(t-\alpha_{\rho_f})}{\phi_f(t-\alpha_{\rho_f})-\phi_f(t-2\alpha_{\rho_f})}+\frac{\rho_fe^{\rho_f t}}{\phi_f(t-\alpha_{\rho_f})-\phi_f(t-2\alpha_{\rho_f})}\right)<A<\infty,
\end{equation}
where the second summand is included only to give an effective upper bound of function $r$
(see \cite[Eq.\,(9.6)]{B}). In particular, in this case the arguments of the proof of \cite[Th.\,2.8]{B} imply that all statements of Theorem A are valid under the assumption
\begin{equation}\label{equ6.34}
\varlimsup_{t\rightarrow\infty}\frac{\phi_f(t+\alpha_{\rho_f})-\phi_f(t-\alpha_{\rho_f})}{\phi_f(t-\alpha_{\rho_f})-\phi_f(t-2\alpha_{\rho_f})}<A<\infty.
\end{equation}
Note that since $\alpha_{\rho_f}\le 1$, using that $\phi_f$ is a convex increasing function we obtain for $s:=t-\alpha_{\rho_f}$,
\begin{equation}\label{equ6.35}
\begin{array}{l}
\displaystyle
\frac{\phi_f(t+\alpha_{\rho_f})-\phi_f(t-\alpha_{\rho_f})}{\phi_f(t-\alpha_{\rho_f})-\phi_f(t-2\alpha_{\rho_f})}\le
\frac{\frac 12 \big(\phi_f(t+1)-\phi_f(t)\bigr)}{\phi_f(s)-\phi_f(s-1)}\medskip\\
\displaystyle = \frac{1}{2}\frac{\phi_f(t+1)-\phi_f(t)}{\phi_f(t)-\phi_f(t-1)}\cdot\frac{\phi_f(s+1)-\phi_f(s)}{\phi_f(s)-\phi_f(s-1)}\cdot\frac{\phi_f(t)-\phi_f(t-1)}{\phi_f(s+1)-\phi_f(s)}\medskip\\
\displaystyle \le \frac{1}{2}\frac{\phi_f(t+1)-\phi_f(t)}{\phi_f(t)-\phi_f(t-1)}\cdot\frac{\phi_f(s+1)-\phi_f(s)}{\phi_f(s)-\phi_f(s-1)}.
\end{array}
\end{equation}
Therefore if $f\in\mathscr C$ and satisfies (cf. \eqref{cI})
\[
\varlimsup_{t\rightarrow\infty}\frac{\phi_f(t+1)-\phi_f(t)}{\phi_f(t)-\phi_f(t-1)}<C,
\]
then due to \eqref{equ6.35} inequality \eqref{equ6.34} is valid with $A:=\frac{C^2}{2}$. This implication and \cite[Th.\,2.8]{B} show that Theorem A is valid for $f\in\mathscr C$ with $\rho_f<\infty$.\smallskip

Now, let us consider the case of $f\in\mathscr C$ with $\rho_f=\infty$. In this instance, the required result does not follow straightforwardly  from Theorem 2.8 of \cite{B} as it is proved under a weaker condition of $\psi=\phi_f$ in \eqref{cII}. 
To prove an analogous result (and therefore our Theorem A) in the general case, one follows the lines of the proof of Theorem 2.8. We just sketch the corresponding arguments leaving the details to the reader.

First, observe that condition (II) in the definition of class $\mathscr C$ implies that
there exists some $v_*\in\mathbb R$ and a continuous function $\kappa : [v_*,\infty)\rightarrow \mathbb R_+$, $\lim_{v\rightarrow\infty}\kappa(v)=1$, such that
\begin{equation}\label{condII}
\ln\psi(v)=\kappa(v)\ln\phi_f(v),\qquad v\ge v_*.
\end{equation}

Now, for each sufficiently large $v\in\mathbb R$ by $s(v)\in\mathbb R_+$ we denote a number such that
\begin{equation}\label{equ6.36}
\frac{\psi(v)}{\psi(v-2s(v))}=e.
\end{equation}
Since $\psi$ is a continuous increasing function,
\[
s(v)=\frac 12 \left(v-\psi^{-1}\bigl(\mbox{$\frac{\psi(v)}{e}$}\bigr)\right),
\]
i.e. $s$ is a continuous in $v$ function. Now, as in the proof of Theorem 2.8 (which uses only monotonicity and convexity of  $\phi_f$) we obtain for $\tilde s(v):=\min(s(v),1)$
\[
\frac{1}{s(v)}\le\frac{2e\psi'(v)}{\psi(v)}
\]
and, cf. \eqref{condII},
\begin{equation}\label{equ6.37}
\frac{1}{\tilde s(v)}\le (\psi(v))^{\frac{\varepsilon(v)}{v}}=(m_f(e^v))^{\frac{\kappa(v)\varepsilon(v)}{v}},\qquad v\ge v_0,
\end{equation}
where $\varepsilon$ is a positive continuous function in $v$ tending to $0$ as $v\rightarrow\infty$ and $v_0 \in \mathbb R$ is sufficiently large.

Next, we determine continuous in $v$ functions
\[
t(v):=e^{\tilde s(v)},\qquad \tilde r(v):=e^{v-\tilde s(v)}.
\]
Then as in the proof of Theorem 2.8 using properties of $\psi$ we obtain (cf. \cite[Eq.\,(8.36)]{B}),
\begin{equation}\label{equ6.38}
\begin{array}{l}
\displaystyle
\ln\tilde r(v)\le \varepsilon'(v)[\ln \psi(v )]^2\le \varepsilon'(v)(\kappa(v))^2[\ln\phi_f(v)]^2\medskip\\
\displaystyle
\quad\qquad = \varepsilon'(v)(\kappa(v))^2 [\ln(m_f(t(v)(\tilde r(v))]^2,\qquad v\ge v_0,
\end{array}
\end{equation}
where $\varepsilon'$ is a positive continuous function in $v$ tending to $0$ as $v\rightarrow\infty$.

Using \eqref{equ6.37} and \eqref{equ6.38} as in \cite[Lm.\,8.5]{B} we have, for all sufficiently large $v$,
\begin{equation}\label{equ6.40}
k(t(v),\tilde r(v))\leq\varepsilon''(v)(\ln m_f(t(v)\tilde r(v)))^2,
\end{equation}
where $\varepsilon''$ is a positive continuous function in $v$ tending to $0$ as $v\rightarrow\infty$, and for such $v$
\begin{equation}\label{equ6.41}
\frac{1}{\ln\left(\frac{1+t(v)}{2\sqrt{t(v)}}\right)}\le 64 \bigl(m_f(t(v)\tilde r(v))\bigr)^{\frac{2\kappa(v)\varepsilon(v)}{v}}.
\end{equation}
Also, for all sufficiently large $v$ using \eqref{condII} and \eqref{equ6.36} we get
(cf. \cite[Eq.\,(8.32)]{B})
\begin{equation}\label{equ6.42}
\begin{array}{l}
\displaystyle
N_f(\tilde r(v), t(v))\ge \frac{\ln\left(\frac{ M_f\left(\frac{\tilde r(v)}{t(v)}\right)}{\sqrt{e}M_f(1)}\right)}{k(t(v),\tilde r(v))}\ge \frac{\left(\psi(v-2\tilde s(v))\right)^{\frac{1}{\kappa(v-2\tilde s(v))}}-\phi_f(0)-\frac 12}{k(t(v),\tilde r(v))}\\
\\
\displaystyle\ge \frac{\bigl(\frac{m_f(\tilde r(v)t(v))}{e}\bigr)^{\frac{\kappa(v)}{\kappa(v-2\tilde s(v))}}-\phi_f(0)-\frac 12}{k(t(v),\tilde r(v))}\ge \frac{(m_f(\tilde r(v)t(v)))^{1-\kappa'(v)}}{k(t(v),\tilde r(v))}, 
\end{array}
\end{equation}
where $\kappa'$ is a positive continuous function in $v$ tending to $0$ as $v\rightarrow\infty$.

Equations \eqref{equ6.40}--\eqref{equ6.42} imply that for all sufficiently large $v\ge v_0$,
\[   
\ln\left(\frac{1+t(v)}{2\sqrt{t(v)}}\right)N_f(\tilde r(v),t(v))\ge (m_f(\tilde r(v)t(v))^{1-\delta(v)}=:k(v)
\]
for some positive continuous function $\delta(v)$ tending to $0$ as $v\rightarrow\infty$. 

Thus, we can proceed as in the proof of Theorem 2.8  of \cite{B} observing that \eqref{equ6.37} gives estimates for $a_1$ and $a_2$ similar to those of the theorem  (cf. \cite[Eq.\,(8.29)]{B}).
\end{proof}

\begin{proof}[{\bf 6.2. Proof of Theorem \ref{theo1.7}}]
For $m=1$ and $P,Q$ the identity automorphisms the required result follows from Theorem A\,(b) of  the previous section by repeating word-for word the arguments of the proof of Theorem \ref{theo1.5} in section 5 (cf. Lemma \ref{lem5.1}).

Next, if $f_j:\mathbb C^{n_j}\rightarrow\mathbb C$, $1\le j\le m$, are of class $\mathscr C$, then due to Theorem \ref{theo1}\,(f) and the above considered case (of $m=1$ and $P,Q$ the identity automorphisms), the graph of $F:=(f_1\times\cdots\times f_m):\mathbb C^{\bar n}\rightarrow\mathbb C^{m}$, $\bar n:=n_1+\cdots +n_m$, satisfies the Bernstein type inequality of exponent $\mu(k)=k^{2+\varepsilon(k)}$, $k\in\mathbb Z_+$, with $\varepsilon$ as in the statement of the theorem. So that for general $P,Q$ the required result follows from the arguments similar to those of Lemma \ref{lem4.1}. 

Specifically, let
$s,t\in\mathbb N$ be such that all coordinates of holomorphic maps $P^{\pm 1}$ and $Q^{\pm 1}$ belong to $\mathcal P_s(\mathbb C^{\bar n})$ and $\mathcal P_t(\mathbb C^m)$, respectively. Then the correspondence $h(z,w)\mapsto h(P^{-1}(z),Q(w))$, $z\in\mathbb C^{\bar n}$, $w\in\mathbb C^m$, determines a linear injective map 
$I:\mathcal P_k(\mathbb C^{\bar n+m})\rightarrow \mathcal P_{k\max(s,t)}(\mathbb C^{\bar n+m})$.
We set
$K:=P\bigl({\rm cl}(\mathbb B^{\bar n})\bigr)$.
Since $P$ is a holomorphic automorphism of $\mathbb C^{\bar n}$, $K$ is a nonpluripolar compact subset of $\mathbb C^{\bar n}$. Then as in the proof of Lemma \ref{lem4.1}  for $\mu(k):=k^{2+\varepsilon(k)}$, $k\in\mathbb Z_+$, with $\varepsilon$  as in the statement of Theorem \ref{theo1.7} and $F_{P,Q}:=Q\circ F\circ P$ we obtain\smallskip
\[
\begin{array}{lr}
\displaystyle
 u^{k}_{{\rm cl}(\mathbb B^{\bar n}),\,\mu}(z;F_{P,Q})\medskip \\
\displaystyle \qquad\qquad\qquad =
\sup\left\{\frac{\ln |h_{F_{P,Q}}(z)|}{\max\bigl(1,k^{2+\varepsilon(k)}\bigr)}\, :\, h\in\mathcal P(\mathbb C^{\bar n+m}),\ {\rm deg}\,h= k,\ \sup_{\mathbb B^{\bar n}}|h_{F_{P,Q}}|=1\right\}\\
\\
\displaystyle \qquad\qquad\qquad =\sup\left\{\frac{\ln |I(h)_{F}(P(z))|}{\max\bigl(1,k^{2+\varepsilon(k)}\bigr)}\, :\, h\in\mathcal P(\mathbb C^{\bar n+m}),\ {\rm deg}\,h= k,\ \sup_{K}|I(h)_{F}|=1\right\}\\
\\
\displaystyle \qquad\qquad\qquad \le \sup\left\{\frac{\ln |g_{F}(P(z))|}{\max\bigl(1,k^{2+\varepsilon(k)}\bigr)}\, :\, g\in\mathcal P(\mathbb C^{\bar n+m}),\ {\rm deg}\,g= k\max(s,t),\ \sup_{K}|g_{F}|=1\right\}\\
\\
\displaystyle \qquad\qquad\qquad\le\frac{\max\bigl(1, (k\max(s,t))^{2+\varepsilon(k\max(s,t))}\bigr)}{\max\bigl(1,k^{2+\varepsilon(k)}\bigr)}\, u_{K,\,\mu}^{k\max(s,t)}(P(z);F)\medskip\\
\displaystyle \qquad\qquad\qquad\le \bigl(\max(s,t)\bigr)^{\, 2+\varepsilon(0)} u_{K,\,\mu}^{k\max(s,t)}(P(z);F),\qquad z\in\mathbb C^{\bar n}.
\end{array}
\]
From here as in \eqref{eq4.20} we get
\[
u_{{\rm cl}(\mathbb B^{\bar n}),\,\mu}(z;F_{P,Q})\le
\bigl(\max(s,t)\bigr)^{2+\varepsilon(0)}\, u_{K,\,\mu} (P(z);F),\quad z\in\mathbb C^{\bar n}.
\]
Since the function  $u_{K,\,\mu}(\cdot\, ; F)$ is locally bounded from above by Theorem \ref{theo1.3}\,(b),  the previous inequality implies that the function $u_{{\rm cl}(\mathbb B^{\bar n}),\,\mu}(\cdot\, ;F_{P,Q})$ is locally bounded from above as well. So by Theorem \ref{theo1.3}\,(b) the graph $\Gamma_{F_{P,Q}}$ admits the Bernstein type inequality of exponent $\mu$. 

This completes the proof of the first statement of the theorem.\smallskip

Next, if all $n_j=1$ and $\rho_{f_j}<\infty$, then in the above arguments $\varepsilon= 0$. In this case,  $\Gamma_{F_{P,Q}}$ admits the Bernstein type inequality of exponent $\mu_{{\rm id}}^2$. This exponent is optimal due to Theorem \ref{theo1}\,(c).

The proof of the theorem is complete.
\end{proof}
\begin{proof}[{\bf 6.3. Proof of Theorem \ref{theo1.8}}]
Without loss of generality we may assume that for some $p\in\{1,\dots,m\}$, 
$\rho_{f_j}<\infty$ if $1\le j\le p$ and $\rho_{f_j}=\infty$ if $p+1\le j\le m$. 

For $g\in\mathcal P_k(\mathbb C^{m+1})$  and $1\le j\le m-1$ we define
\[
g_j(z, \mathbf z_j):=
g(z,f_1(z),\dots, f_j(z), \mathbf z_j), \quad z\in\mathbb C,\
\mathbf z_j:=(z_{j+1},\dots , z_{m})\in\mathbb C^{m-j}.
\]
Also, we set
\[
\begin{array}{r}
g_0(z, \mathbf z_0):=
g(z,\mathbf z_0)\quad {\rm and}\quad  g_{m}(z):=g(z,f_1(z),\dots, f_{m}(z)),\medskip\\
\quad z\in\mathbb C,\  \mathbf z_0:=(z_1,\dots, z_{m})\in\mathbb C^{m}.
\end{array}
\]

By definition, $g_j$ is an entire function on $\mathbb C^{m-j+1}$ such that for each fixed $z\in\mathbb C$ function $g_j(z,\cdot)\in\mathcal P_k(\mathbb  C^{m-j})$. In what follows, we add index $_{f_j}$ to all characteristics of Theorem A of section 6.1 related to the function $f:=f_j$ (e.g., $r:=r_{f_j}$, $\varepsilon:=\varepsilon_{f_j}$, etc).

Theorem \ref{theo1.8} is the direct consequence of the following result.
\begin{Th}\label{theo6.3}
There exist numbers $C_j\in\mathbb R_+$, $k_j\in\mathbb N$ and converging to zero sequences $\{\varepsilon_j(k)\}_{k\ge k_j}\subset\mathbb R_+$, $0\le j\le m$,  such that  $\varepsilon_j= 0$ for $0\le j\le p$ and for all $g\in\mathcal P_k(\mathbb C^{m+1})$ with $k\ge k_j$ and all
$\mathbf z_j\in\mathbb C^{m-j}$
\[
g_j(\cdot, \mathbf z_j)\in\mathcal F_{p_j(k),k}\bigl(er_{f_{j+1}}(p_{j}(k));e;\infty\bigr),\quad {\rm where}\quad
p_j(k):=C_j k^{2^j+\varepsilon_j(k)},\quad 0\le j\le m
\]
(here we set $r_{f_{m+1}}:=r_{f_{m}}$).
\end{Th}
\begin{proof}
We define $C_0=1$, $k_0=1$, $\varepsilon_0= 0$ and 
prove the result by induction on $j$.

For $j=0$ the function $g_0:=g\in \mathcal P_k(\mathbb C^{m+p+1})$. In particular, $g_0\in\mathcal F_{k,k}(er;e;\infty)$ for all positive numbers $r$ and, hence, for $r=r_{f_1}(k)$. This establishes the base of induction.

Next, assuming that the result holds for $0<j<m$, let us prove it for $j+1$. 

To this end, we apply Theorem A\,(b) to functions $f:=f_{j+1}$ and $g(z,w):=g_j(z,w,\mathbf z_{j+1})$, $(z,w)\in\mathbb C^2$, with $p=k$ equal to $p_j(k)$. Then, by the induction hypothesis, one derives from the theorem that for all $g\in \mathcal P_k(\mathbb C^{m+1})$ with $k$ such that $p_j(k)\ge \tilde k_{j+1}:=\max\bigl(k_{0f_{j+1}}, r_{f_{j+1}}^{-1}(r_{0f_{j+1}})\bigr )$,\smallskip
\begin{equation}\label{equat6.44}
\begin{array}{l}
g_{j+1}(\cdot,\mathbf z_{j+1})\in\mathcal F_{\tilde p_{j+1}(k),k}(r_{f_{j+1}}(p_j(k));e;\infty),\quad {\rm where}\medskip\\
\displaystyle\quad\qquad\qquad\qquad \tilde p_{j+1}(k):=C_{f_{j+1}}(p_{j}(k))^{2+\varepsilon_{f_{j+1}}(p_j(k))}.
 \end{array}
 \end{equation}

Next, by the definitions of $p_{j}(k)$ and $\varepsilon_{f_{j+1}}$,
\[
\begin{array}{l}
\displaystyle \tilde p_{j+1}(k)=C_{f_{j+1}} C_j^{2+\varepsilon_{f_{j+1}}(p_j(k))} k^{2^{j+1}+2^j \varepsilon_{f_{j+1}}(p_j(k))+\varepsilon_j(k) (2+\varepsilon_{f_{j+1}}(p_j(k)))}\medskip\\
\displaystyle
\quad\qquad\ \le C_{j+1} k^{2^{j+1}+\varepsilon_{j+1}(k)}=:p_{j+1}(k).\quad\qquad\
\end{array}
\]
Here 
\[
\begin{array}{l}
\displaystyle
\varepsilon_{j+1}(k):=2^j \varepsilon_{f_{j+1}}(p_j(k))+\varepsilon_j(k) (2+\varepsilon_{f_{j+1}}(p_j(k)))\quad {\rm and}\medskip\\
\displaystyle
\quad C_{j+1}:=\sup_{k\ge\tilde k_{j+1}}\left\{C_{f_{j+1}} C_j^{2+\varepsilon_{f_{j+1}}(p_j(k))}\right\}.
\end{array}
\]
(The number is finite because $\varepsilon_{f_{j+1}}$ is a bounded function.)

Note that the above expression for $\varepsilon_{j+1}$ and statement (1) of Theorem A (applied to functions $f_i$ for $1\le i\le m$) show that $\varepsilon_{j+1}= 0$ whenever $j+1\le m$.  For other indices, $\lim_{k\rightarrow\infty}\varepsilon_{j+1}(k)=0$ (as $\varepsilon_j$ and $\varepsilon_{f_{j+1}}$ possess this property and $p_j(k)\rightarrow\infty$ as $k\rightarrow\infty$).

To complete the proof of the inductive step we must show that for all sufficiently large integers $k$ and all $j+2\le m$
\begin{equation}\label{eq6.31}
er_{f_{j+2}}(p_{j+1}(k))\le  r_{f_{j+1}}(p_j(k)).
\end{equation}
 
To establish this fact we consider three cases.\smallskip

(1) $j+2\le p$.  In this case $f_{j+1}$ and $f_{j+2}$ satisfy condition (I), see \eqref{cI}. Also, due to equation \eqref{equ6.32}, see \cite[Eq.\,(9.3), (9.7)]{B}, functions $r_{f_{s}}$ are right inverses of nondecreasing functions
\[
k_{f_{s}}(r):=\frac{m_{f_s}(e^{-\alpha_{\rho_{f_s}}}r)-m_{f_s}(e^{-2\alpha_{\rho_{f_s}}}r)-1}{9(\sqrt{e}+1)^2(\rho_{f_s}^2+1)(17+2\ln(\rho_{f_s}+1))},\quad r\ge r_{0f_s},\quad s=j+1,\ j+2.
\]
Since $\varepsilon_{s}= 0$, by the definition of $p_s(k)$, $s=j+1,\, j+2$,  for all sufficiently large integers $k$,
\[
er_{f_{j+2}}(p_{j+1}(k))= er_{f_{j+2}}(C_{j+1} k^{2^{j+1}})\quad {\rm and}\quad 
r_{f_{j+1}}(p_j(k))= r_{f_{j+1}}(C_j k^{2^{j}}).
\]
Passing here to inverse functions we reduce \eqref{eq6.31} to the question on the validity, for all sufficiently large $r$, of the inequality
\[
\left(\frac{k_{f_{j+2}}\left(\frac re\right)}{C_{j+1}}\right)^{\frac{1}{2^{j+1}}}\ge 
\left(\frac{k_{f_{j+1}}(r)}{C_{j}}\right)^{\frac{1}{2^{j}}}.
\]
In turn, the latter inequality is the consequence of the following result.
\begin{Lm}\label{lem6.4}
Under the hypotheses of Theorem \ref{theo1.8}, see \eqref{eq1.4},
\[
 \lim_{r\rightarrow\infty}\frac{(k_{f_{j+1}}(r))^2}{k_{f_{j+2}}\left(\frac re\right)}=0.
\]
\end{Lm}
\begin{proof}
Making use of explicit expressions for functions $k_{f_s}$ we obtain
\begin{equation}\label{eq6.32}
\lim_{r\rightarrow\infty}\frac{(k_{f_{j+1}}(r))^2}{k_{f_{j+2}}\left(\frac re\right)}
=\lim_{r\rightarrow\infty}\frac{ \bigl(m_{f_{j+1}}(e^{-\alpha_{\rho_{f_{j+1}}}}r)-m_{f_{j+1}}(e^{-2\alpha_{\rho_{f_{j+1}}}}r)\bigr)^{2}}{m_{f_{j+2}}(e^{-\alpha_{\rho_{f_{j+2}}-1}} r)-m_{f_{j+2}}(e^{-2\alpha_{\rho_{f_{j+2}}-1}} r)}.
\end{equation}
Since $\alpha_{\rho_f}\le 1$, see \eqref{equ6.32}, by the maximum principle for subharmonic functions
\begin{equation}\label{eq6.33}
m_{f_{j+1}}(e^{-\alpha_{\rho_{f_{j+1}}}}r)-m_{f_{j+1}}(e^{-2\alpha_{\rho_{f_{j+1}}}}r)\le m_{f_{j+1}}(r)-m_{f_{j+1}}\left(\mbox{$\frac re$}\right).
\end{equation}

Next, assume that $f_{j+2}\in \mathscr C$ and the limit superior in equation \eqref{cI} for $f_{j+2}$ is bounded from above by a constant $C$. Then due to \eqref{equ6.35} for all sufficiently large $r$,
\begin{equation}\label{equ6.46}
\frac{m_{f_{j+2}}
(e^{\alpha_{\rho_{f_{j+2}}}}r)-m_{f_{j+2}}(e^{-\alpha_{\rho_{f_{j+2}}}}r)}{m_{f_{j+2}}(e^{-\alpha_{\rho_{f_{j+2}}}}r)-m_{f_{j+2}}(e^{-2\alpha_{\rho_{f_{j+2}}}}r)}<A:=\frac{C^2}{2}.
\end{equation}
Applying inequality \eqref{equ6.46} $\ell+1$ times, $\ell:=\bigl\lceil \frac{1}{\alpha_{\rho_{f_{j+2}}}}\bigr\rceil$,  and after that the maximum  principle for subharmonic functions,  for all sufficiently large $r$ we obtain\smallskip
\begin{equation}\label{eq6.34}
\begin{array}{l}
\displaystyle
m_{f_{j+2}}(e^{-\alpha_{\rho_{f_{j+2}}-1}}\, r)-m_{f_{j+2}}(e^{-2\alpha_{\rho_{f_{j+2}}-1}}\, r)>\frac{m_{f_{j+2}}(e^{-1}\, r)-m_{f_{j+2}}(e^{-\alpha_{\rho_{f_{j+2}}-1}}\, r)}{A}\medskip\\
\displaystyle >\frac{m_{f_{j+2}}(e^{\alpha_{\rho_{f_{j+2}}-1}}\, r)-m_{f_{j+2}}(e^{-1}\, r)}{A^2}>\cdots >\frac{m_{f_{j+2}}(e^{\ell\, \alpha_{\rho_{f_{j+2}}-1}}\, r)-m_{f_{j+2}}(e^{(\ell-1)\, \alpha_{\rho_{f_{j+2}}-1}}\, r)}{A^{\ell+1}}\medskip\\
\displaystyle\ge \frac{m_{f_{j+2}}( r)-m_{f_{j+2}}(\frac re)}{A^{\ell+1}}.
\end{array}
\end{equation}
Using inequalities \eqref{eq6.33}, \eqref{eq6.34} in the right-hand side of \eqref{eq6.32}, due to condition \eqref{eq1.4} of the theorem, we get
\[
\lim_{r\rightarrow\infty}\frac{(k_{f_{j+1}}(r))^2}{k_{f_{j+2}}\left(\frac re\right)}\le
\lim_{r\rightarrow\infty}\frac{A^{\ell+1}\cdot\bigl(m_{f_{j+1}}(r)-m_{f_{j+1}}\left(\frac re\right)
 \bigr)^{2}}{m_{f_{j+2}}( r)-m_{f_{j+2}}(\frac re)}=0.
\]

The proof of the lemma is complete.
\end{proof}

As we explained earlier, Lemma \ref{lem6.4} implies inequality \eqref{eq6.31} for all sufficiently large integers $k$. Now, as $k_{j+1}\in\mathbb N$ in the theorem we choose a natural number such that $p_j(k_{j+1})\ge \tilde k_{j+1}$ and that \eqref{eq6.31} is valid for all integers $k\ge k_{j+1}$. 

This completes the proof of the inductive step in case (1).\smallskip

(2) $j+1=p$. In this case $f_{j+1}$ satisfies condition (I)  and $f_{j+2}$ satisfies condition (II), see \eqref{cI}, \eqref{cII}. Thus, as before, for all sufficiently large $k$,
$r_{f_{j+1}}(p_{j}(k))=r_{f_{j+1}}(C_j k^{2^j})$  and, due to Theorem A part (2),
\[
r_{f_{j+2}}(p_{j+1}(k))=\frac{1}{e^2}m_{f_{j+2}}^{-1}\left(\left(C_{j+1} k^{2^{j+1}}\right)^{1+\varepsilon''_{f_{j+2}}(C_{j+1} k^{2^{j+1}})}\right),
\]
where the nonnegative function $\varepsilon''_{f_{j+2}}$ decreases to zero.

Next, by the definition of function $k_{f_{j+1}}$, for all sufficiently large $r$,
\[
k_{f_{j+1}}(r)\le m_{f_{j+1}}(r).
\]
Passing here to right inverse functions we get, for all sufficiently large $k$,
\[
r_{f_{j+1}}(k)\ge m_{f_{j+1}}^{-1}(k).
\]
Using these facts we conclude that in order to establish inequality \eqref{eq6.31} in this case, it suffices to prove that for all sufficiently large $k$
\[
m_{f_{j+2}}^{-1}\left(\left(C_{j+1} k^{2^{j+1}}\right)^{2}\right)\le m_{f_{j+1}}^{-1}(C_j k^{2^j}).
\]
The latter can be derived from the following result by passing to inverse functions.
\begin{Lm}\label{lem6.5}
For all sufficiently large $r$,
\[
\left(\frac{1}{C_j}m_{f_{j+1}}(r)\right)^{2^{-j}}\le \left(\frac{1}{C_{j+1}^2}m_{f_{j+2}}(r)\right)^{2^{-j-2}}.
\]
\end{Lm}
\begin{proof}
We apply condition \eqref{cII} for $f_{j+2}$ assigning index $_{j+2}$ to all functions which appear there. According  to this condition, for each $\varepsilon>0$ there exists some $t_\varepsilon>0$ such that for all $t\ge t_{\varepsilon}$,
\[
-\left(\frac{1}{\ln\psi_{j+2}(t)}\right)'<\frac{\varepsilon}{t^2}.
\]
(The minus sign reflects the fact that the derivative of the function is nonpositive.)
 
Integrating this inequality from $t$ to infinity we get
\[
\frac{1}{\ln\psi_{j+2}(t)}<\frac{\varepsilon}{t},\qquad t\ge t_\varepsilon .
\]
Due to condition (II) for $f_{j+2}$ this implies, for all sufficiently large $t$,
\[
\ln m_{f_{j+2}}(e^t)\ge \frac{\ln\psi_{j+2}(t)}{2}>\frac{t}{2\varepsilon}.
\]
Let us choose here $\varepsilon:=\frac{1}{8(\rho_{f_{j+1}}+1)}$.
Then from the previous inequality and the fact that $f_{j+1}$ is of finite order $\rho_{f_{j+1}}$ we obtain
\[
\varlimsup_{r\rightarrow\infty}\frac{\ln m_{f_{j+1}}(r)}{\ln r}=\rho_{f_{j+1}}<\rho_{f_{j+1}}+1=\frac{1}{8\varepsilon}\le
\varliminf_{r\rightarrow\infty}\frac{\ln m_{f_{j+2}}(r)}{4\ln r}.
\]
This implies the required statement of the lemma.
 \end{proof}
 Now, choosing $k_{j+1}\in\mathbb N$ as at the end of the proof of case (1) we complete the proof of the inductive step in case (2).\smallskip

(3) $p<j+1$. In this case $f_{j+1}$ and $f_{j+2}$ satisfy condition (II), see \eqref{cII}. Thus, as before, for all sufficiently large $k$ and $s=j+1,\, j+2$,
\[
r_{f_{s}}(p_{s-1}(k))=\frac{1}{e^2}m_{f_{s}}^{-1}\left(\left(C_{s-1} k^{2^{s-1}}\right)^{1+\varepsilon''_{f_{s}}(C_{s-1} k^{2^{s-1}})}\right),
\]
where the nonnegative functions $\varepsilon''_{f_{s}}$ decrease to zero.

Hence, to establish inequality \eqref{eq6.31} in this case we must show that for all sufficiently large integers $k$,
\[
e\, m_{f_{j+2}}^{-1}\left(\left(C_{j+1} k^{2^{j+1}}\right)^{1+\varepsilon''_{f_{j+2}}(C_{j+1} k^{2^{j+1}})}\right)\le m_{f_{j+1}}^{-1}\left(\left(C_{j} k^{2^{j}}\right)^{1+\varepsilon''_{f_{j+1}}(C_{j} k^{2^{j}})}\right).
\]
This inequality follows straightforwardly from the next result.
\begin{Lm}\label{lem6.6}
There exists some $\alpha>1$ such that for all sufficiently large $k$,
\[
e\, m_{f_{j+2}}^{-1}\left(\left(C_{j+1} k^{2^{j+1}}\right)^{\alpha}\,\right)\le m_{f_{j+1}}^{-1}\left(C_{j} k^{2^{j}}\right).
\]
\end{Lm}
\begin{proof}
Passing to inverse functions we rewrite the required inequality as the inequality
\[
\left(\frac{1}{C_j}m_{f_{j+1}}(r)\right)^{2^{-j}}\le \left(\frac{1}{C_{j+1}^{\alpha}}\,m_{f_{j+2}}\left(\frac re\right)\right)^{\frac{2^{-j-1}}{\alpha}}
\]
valid for all sufficiently large $r>0$.

This is true if 
\[
\varlimsup_{r\rightarrow\infty}\frac{\ln m_{f_{j+1}}(r)}{\ln m_{f_{j+2}}(\frac re)}<\frac{1}{2\alpha}.
\]
But according to  condition \eqref{eq1.5} of the theorem
\[
\varlimsup_{r\rightarrow\infty}\frac{\ln m_{f_{j+1}}(r)}{\ln m_{f_{j+2}}(\frac re)}=L<\frac{1}{2}.
\]
Hence, it suffices to choose
\[
\alpha=\frac{2}{1+2L}.
\]
This completes the proof of the lemma.
\end{proof}
Finally, choosing $k_{j+1}\in\mathbb N$ as in cases (1),\,(2) we complete the proof of the inductive step.

Thus Theorem \ref{theo6.3} is proved by induction on $j$. 
\end{proof}

Applying Theorem \ref{theo6.3} with $j=m$, by the definition of class $\mathcal F_{p_{m}(k),k}(er_{f_{m}}(p_{m}(k);e;\infty)$ repeating the arguments of the proof of Theorem \ref{theo1.5} (see Lemma \ref{lem5.1}) we obtain that $\Gamma_f$ admits the Bernstein type inequality of exponent $\mu(k):=k^{2^{m}+\varepsilon(k)}$, $k\in\mathbb Z_+$, for some $\varepsilon:\mathbb Z_+\rightarrow\mathbb R_+$ decreasing to zero.

This completes the proof of Theorem \ref{theo1.8}.
\end{proof}
\begin{R}\label{rem6.6}
{\rm Arguing as in the proof of \cite[Th.\,2.3]{B}, one deduces directly from Theorem \ref{theo6.3} by means of Theorem A\,(a) that there exist some constant $C\in\mathbb R_+$, a number $k_*\in\mathbb N$ and a  decreasing to $0$ continuous function $\varepsilon_*: (k_*,\infty)\rightarrow\mathbb R_+$ equal to $0$ if all $\rho_{f_j}=0$ such that for all $p\in\mathcal P_k(\mathbb C^{m+1})$ with $k\ge k_*$ 
\begin{equation}
\max_{\mathbb D^{m+1}} |p|\le e^{C k^{2^{m}+\varepsilon_*(k)}\ln\bigl(r_{f_1}(p_0(k))\cdots r_{f_m}(p_{m-1}(k))\bigr)}\max_{\mathbb D}|p_f|;
\end{equation}\label{equ6.50}
recall that $p_f(z):=p(z,f_1(z),\dots, f_m(z))$, $z\in\mathbb C$, also, $\mathbb D^{m+1}:=\times^{m+1}\,\mathbb D$.

Here, according to \eqref{eq6.31}, for a sufficiently large $k_*$ and all $k\ge k_*$
\[
\ln\bigl(r_{f_1}(p_0(k))\cdots r_{f_m}(p_{m-1}(k))\bigr)\le m\ln\bigl(r_{f_1}(k)\bigr).
\]
Moreover, cf. \cite[Th.\,2.8]{B}, for all such $k$,
\[
r_{f_1}(k)\le\left\{
\begin{array}{ccc}
c_{f_1}k^{\frac{1}{\rho_{f_1}}}&{\rm if}&0<\rho_{f_1}<\infty\medskip\\
k^{\delta_{f_1}(k)}&{\rm if}&0<\rho_{f_1}=\infty
\end{array}
\right.
\]
for a constant $c_{f_1}\in\mathbb R_+$ and a decreasing to $0$ nonnegative function $\delta_{f_1}\in C([k_*,\infty))$.
}
\end{R}

\section{Proofs of Propositions \ref{propo1.11}, \ref{prop1.9} and \ref{prop1.10}}
\begin{proof}[{\bf 7.1. Proof of Proposition \ref{propo1.11}}]
(1) For $\rho_f<\infty$ the assumption of the proposition implies,
for some constant $c>0$,
\begin{equation}\label{equa7.49}
\phi_f(t)-c\le\phi_g(t)\le\phi_f(t)+c,\qquad t\in\mathbb R.
\end{equation}
Hence,
\[
\varlimsup_{t\rightarrow\infty}\frac{\phi_g(t+1)-\phi_g(t)}{\phi_g(t)-\phi_g(t-1)}\le \varlimsup_{t\rightarrow\infty}\frac{\phi_f(t+1)-\phi_f(t)+2c}{\phi_f(t)-\phi_f(t-1)-2c}=\varlimsup_{t\rightarrow\infty}\frac{\phi_f(t+1)-\phi_f(t)}{\phi_f(t)-\phi_f(t-1)}<\infty,
\]
i.e. $g\in\mathscr C$. (Here the last equality is due to the fact that  $\phi_f(t+1)-\phi(t)$, $t\in\mathbb R$, is a nondecreasing unbounded from above function as $f$ is nonpolynomial.) 

For $\rho_f=\infty$, inequality \eqref{equa7.49} yields
\[
\lim_{t\rightarrow\infty}\frac{\ln\phi_g(t)}{\ln\phi_f(t)}=1.
\]
Thus, conditions \eqref{cII} for $f$  and $g$ coincide, i.e. $g\in\mathscr C$.\smallskip

\noindent (2) The assumption of the proposition leads to the inequality: 
\[
\frac 1c \phi_f(t)\le \phi_g(t)\le c\phi_f(t),\qquad t\in\mathbb R,
\]
for some constant $c>1$.

In turn, this implies
\[
\lim_{t\rightarrow\infty}\frac{\ln\phi_g(t)}{\ln\phi_f(t)}=1.
\]
Thus, as above, conditions \eqref{cII} for $f$ and $g$ coincide, i.e. $g\in\mathscr C$.\smallskip
 
\noindent (3) The statement holds true because $m_{f^n}=n\, m_f$ for all $n\in\mathbb N$.
\end{proof}

\begin{proof}[{\bf 7.2. Proof of Proposition \ref{prop1.9}}]
Let
\[
\varliminf_{r\rightarrow\infty}\frac{m_f(r)}{r^{\rho(r)}}=\mu_f>0.
\]
Recall that the proximate order satisfies the following property, see, e.g., \cite{L}: {\em for every $\varepsilon>0$ and every $0<a<b<\infty$
there exists $r_0>0$ such that for all $k\in [a,b]$ and $r\ge r_0$,}
\begin{equation}\label{e7.36}
(1-\varepsilon)k^{\rho_f}r^{\rho(r)}<(kr)^{\rho(kr)}<(1+\varepsilon)k^{\rho_f}r^{\rho(r)}.
\end{equation}

Let us fix some $\varepsilon\in\bigl(0, \min\bigl(\frac 12,\frac{\mu_f}{2}\bigr)\bigr)$ and define a number $d>0$ from the identity
\begin{equation}\label{e7.37}
\frac{\mu_f}{2}-3\sigma_f e^{-\rho_f d}=\frac{9\sigma_f -\mu_f e^{-\rho_f}}{4}\cdot \frac{2\mu_f}{9\sigma_f}.
\end{equation}
Then, due to \eqref{e7.36}, convexity of $\phi_f$ and definitions of $\mu_f$ and $\sigma_f$,  there exists some $t_\varepsilon\in\mathbb R$ such that for all $t\ge t_\varepsilon$, 
\[
\begin{array}{l}
\displaystyle  \phi_f(t+1)-\phi_f(t) \le 
(\sigma_f+\varepsilon)e^{(t+1)\rho(e^{t+1})}-(\mu_f-\varepsilon)e^{t\rho(e^{t})}\medskip\\
\displaystyle \ \ \quad\qquad\qquad\qquad\le (\sigma_f+\varepsilon)(1+\varepsilon)e^{2\rho_f}e^{(t-1)\rho(e^{t-1})} -(\mu_f-\varepsilon)(1-\varepsilon)e^{\rho_f }e^{(t-1)\rho(e^{t-1})}\medskip\\
\displaystyle \ \ \quad\qquad\qquad\qquad \le \frac{9\sigma_f e^{\rho_f}-\mu_f}{4}\, e^{(t-1)\rho(e^{t-1})}e^{\rho_f}=\frac{9\sigma_f e^{2\rho_f}d}{2\mu_f}\cdot \frac{\frac{\mu_f}{2}-3\sigma_f e^{-d\rho_f}}{d}\,e^{(t-1)\rho(e^{t-1})}\medskip\\
\displaystyle 
\ \ \quad\qquad\qquad\qquad\le \frac{9\sigma_f e^{2\rho_f}d}{2\mu_f}\cdot
\frac{(\mu_f-\varepsilon)e^{(t-1)\rho(e^{t-1})}-\frac{\sigma_f+\varepsilon}{1-\varepsilon}e^{-d\rho_f }e^{(t-1)\rho(e^{t-1})}}{d}\medskip\\
\displaystyle \ \ \quad\qquad\qquad\qquad
\le \frac{9\sigma_f e^{2\rho_f}d}{2\mu_f}\cdot\frac{(\mu_f-\varepsilon)e^{(t-1)\rho(e^{t-1})}-(\sigma_f+\varepsilon)e^{(t-1-d)\rho(e^{t-1-d})}}{d}\medskip\\
\displaystyle \ \ \quad\qquad\qquad\qquad
\le\frac{9\sigma_f e^{2\rho_f}d}{2\mu_f}\cdot\frac{\phi_f(t-1)-\phi_f(t-1-d)}{d}
\le \frac{9\sigma_f e^{2\rho_f}d}{2\mu_f}\,\bigl(\phi_f(t)-\phi_f(t-1)\bigr).
\end{array}
\]

This shows that $f$ satisfies \eqref{cI}, i.e. $f\in\mathscr C$.
\end{proof}

\smallskip
\begin{proof}[{\bf 7.3. Proof of Proposition \ref{prop1.10}}]
We set $u:={\rm Re}\, f$ and $h:=e^f$. Then $M_h(r)=M_{e^u}(r)$ and so $m_h(r)=M_u(r)$,
$r>0$. To prove that $h\in\mathscr C$, we consider two cases.

First, assume that $f\in\mathscr C$ satisfies conditions (I), see \eqref{cI}, and \eqref{eq1.11}. We prove that 
$h\in\mathscr C$ satisfies condition (II) with $\psi(t)=\phi_{h}(t)\, (:=m_h(e^t)=M_u(e^t))$, see \eqref{cII}: 
\begin{Lm}\label{lem7.2}
\[
\lim_{t\rightarrow\infty}t^2\left(\frac{1}{m_{u}(e^t)}\right)'=0.
\]
\end{Lm}
\begin{proof}
Applying the Borel-Carath\'eodory theorem to $f$ restricted to each complex line passing through the origin we obtain, for $0<s<1$ and all $r>0$,
\begin{equation}\label{eq7.35}
M_f(sr)\le \frac{2}{(1-s)}\,M_u(r)+\frac{1+s}{1-s}\, |f(0)|.
\end{equation}
On the other hand, obviously
\begin{equation}\label{eq7.36}
M_u(r)\le M_f(r).
\end{equation}
Further, condition \eqref{cI} for $f$ implies, for some $A>0$ and all sufficiently large $t$,
\[
\phi_f(t+1)-\phi_f(t)<A\,\bigl(\phi_f(t)-\phi_f(t-1)\bigr)<A\phi_f(t).
\]
Hence, for such $t$,
\begin{equation}\label{eq7.37}
\phi_f(t+1)<c_1 \phi_f(t),\qquad c_1:=A+1.
\end{equation}

Thus, for $q(t):= m_{u}(e^t)$, $t\in\mathbb R$, from \eqref{eq7.35} with $s=\frac 1e$, \eqref{eq7.36} and \eqref{eq7.37} we obtain, for all sufficiently large $t$,
\begin{equation}\label{eq7.38}
q(t+1)\le \phi_f(t+1)<c_1^2\phi_f(t-1)\le c_2q(t),
\end{equation}
for some $c_2$ depending on $A$ and $f$.

Now, from \eqref{eq7.38} using that $q'$ is a nonnegative nondecreasing function by \eqref{eq1.11} we obtain
\[
\begin{array}{l}
\displaystyle
0\le \varliminf_{t\rightarrow\infty}t^2\left(-\frac{1}{ q(t)}\right)'\le\varlimsup_{t\rightarrow\infty}t^2\left(-\frac{1}{ q(t)}\right)'= \varlimsup_{t\rightarrow\infty}\frac{t^2\, q'(t)}{q^2(t)}\le  \varlimsup_{t\rightarrow\infty}\frac{t^2\bigl(q(t+1)-q(t)\bigr)}{q^2(t)}\medskip\\
\displaystyle \quad\le \varlimsup_{t\rightarrow\infty}\frac{t^2 q(t+1)}{  q^2(t)}\le\varlimsup_{t\rightarrow\infty}\frac{t^2\,c_2\,q(t)}{q^2(t)}\le\varlimsup_{t\rightarrow\infty}\frac{c_2^2 \,t^2}{\phi_f(t)}=0.
\end{array}
\]
\end{proof}

Next, we consider the case of $f\in\mathscr C$ satisfying condition (II), see \eqref{cII}.  Let us prove the following result.
\begin{Lm}\label{lemma7.2}
Function $f$ satisfies condition \eqref{eq1.11}.
\end{Lm}
\begin{proof}
Due to condition \eqref{cII}, for each $\varepsilon>0$ there exists some $t_\varepsilon>0$ such that for all $t\ge t_\varepsilon$
\[
-\left(\frac{1}{\ln\psi(t)}\right)'\le \frac{\varepsilon}{t^2}.
\]
Integrating this inequality from $t$ to $\infty$ we get for all $t\ge t_\varepsilon$
\[
\frac{t}{\ln\psi(t)}\le \varepsilon.
\]
This implies that
\begin{equation}\label{eq7.54}
\lim_{t\rightarrow\infty}\frac{t}{\ln\psi(t)}=0.
\end{equation}
Further, due to condition (II) there exists some $t_*\in\mathbb R$ and a continuous function $\kappa : [t_*,\infty)\rightarrow \mathbb R_+$, $\lim_{t\rightarrow\infty}\kappa(t)=1$, such that
\begin{equation}\label{eq7.55}
\ln\psi(t)=\kappa(t)\ln\phi_f(t),\qquad t\ge t_*.
\end{equation}
From here and \eqref{eq7.54} we obtain
\[
0\le\varliminf_{t\rightarrow\infty}\frac{t^2}{\phi_f(t)}\le \varlimsup_{t\rightarrow\infty}\frac{t^2}{\phi_f(t)}\le \varlimsup_{t\rightarrow\infty}\frac{t^2}{(\ln\phi_f(t))^2}=
 \varlimsup_{t\rightarrow\infty}\frac{t^2}{(\ln\psi(t))^2}=0.
\]
That is, $f$ satisfies condition \eqref{eq1.11}.
\end{proof}
Due to the lemma, without additional restrictions on $f$, we must show that $e^f$ satisfies condition (II), see \eqref{cII}.  

We proceed as in the proof of Theorem~A  denoting by $s: [t_0,\infty)\rightarrow\mathbb R$, for some $t_0>0$, a continuous function
such that
\begin{equation}\label{eq7.56}
\frac{\psi(t)}{\psi(t-2s(t))}=e.
\end{equation}
Then for $\tilde s(t):=\min(s(t),1)$,
\begin{equation}\label{eq7.57}
\frac{1}{\tilde s(t)}\le \bigl(m_f(e^t)\bigr)^{\frac{\kappa(t)\varepsilon(t)}{t}},\qquad t\ge t_0,
\end{equation}
where $\varepsilon\in C\bigl([t_0,\infty)\bigr)$ is a positive continuous function tending to zero at $\infty$.

Next, according to \eqref{eq7.35}, \eqref{eq7.36}, for all  $t\ge t_0$ for a sufficiently large $t_0$,
\[
\phi_f(t-2\tilde s(t))\le \ln\frac{\tilde c}{\tilde s(t)}+m_u(e^t)\qquad {\rm and}\qquad m_u(e^t)\le\phi_f(t)
\]
for an absolute constant  $\tilde c>0$. (Recall that $u:={\rm Re}\, f$.)

From here, due to \eqref{eq7.55},  \eqref{eq7.56}, \eqref{eq7.57} we obtain, for all $t\ge t_0$, 
\begin{equation}\label{eq7.58}
(1-\delta(t))\bigl(e^{-1}\psi(t)\bigr)^{\frac{1}{\kappa(t-2\tilde s(t))}}\le (1-\delta(t))\,\phi_f(t-2\tilde s(t) )\le m_u(e^t)\le \phi_f(t),
\end{equation}
where $\delta\in C\bigl([t_0,\infty)\bigr)$ is a positive continuous function tending to zero at $\infty$.

Now, function ${\rm id}-2\tilde s\in C\bigl([t_0,\infty)\bigr)$ tends to $\infty$ at $\infty$ and has minimal value in the interval $[t_0-2, t_0)$. In particular, each $t\ge t_0$ can be written as $t=v_t-2\tilde s(v_t)$ for some $v_t> t$. So for a sufficiently large $t_0>0$ and all $t\ge t_0$ we have by \eqref{eq7.58}, \eqref{eq7.56} and \eqref{eq7.57},
\[
\begin{array}{l}
\displaystyle
m_u(e^t)=m_u\bigl(e^{v_t-2\tilde s(v_t)}\bigr)\ge 
(1-\delta(t))\bigl(e^{-1}\psi(v_t-2\tilde s(v_t))\bigr)^{\frac{1}{\kappa(t-2\tilde s(t))}}\medskip\\
\displaystyle \quad \qquad\ge (1-\delta(t))\bigl(e^{-2}\psi(v_t)\bigr)^{\frac{1}{\kappa(t-2\tilde s(t))}}=(1-\delta(t))e^{-\frac{2}{\kappa(t-2\tilde s(t))}}\bigr(\phi_f(v_t)\bigr)^{\frac{\kappa(v_t)}{\kappa(t-2\tilde s(t))}}\ge
\bigl(\phi_f(v_t)\bigr)^{\frac 34}\medskip\\
\displaystyle \quad\qquad\ge \bigl(m_u(e^{v_t})\bigr)^{\frac 34}\ge \bigl(m_u(e^{v_t})-m_u\bigl(e^{v_t-2\tilde s(v_t)}\bigr)\bigr)^{\frac 34}
\ge \bigl(m_u'\bigl(e^{v-2\tilde s(v)}\bigr)\bigr)^{\frac 34}(2\tilde s(v))^{\frac 34}\medskip\\
\displaystyle \qquad\quad=\bigl(m_u'(e^{t})\bigr)^{\frac 34} (2\tilde s(v))^{\frac 34}\ge \bigl(m_u'(e^{t})\bigr)^{\frac 34}\bigl(m_f(e^t)\bigr)^{-\frac{3\kappa(t)\varepsilon(t)}{4t}}\ge \bigl(m_u'(e^{t})\bigr)^{\frac 34}\bigl(m_f(e^t)\bigr)^{-\frac{1}{4}} .
\end{array}
\]
From the previous inequality and equations \eqref{eq7.58}, \eqref{eq7.54}, \eqref{eq7.55} we obtain
\[
\begin{array}{l}
\displaystyle
0\le \varliminf_{t\rightarrow\infty}t^2\left(-\frac{1}{\ln \phi_h(t)}\right)'\le \varlimsup_{t\rightarrow\infty}t^2\left(-\frac{1}{\ln \phi_h(t)}\right)'= \varlimsup_{t\rightarrow\infty}t^2\left(-\frac{1}{m_u(e^t)}\right)'=\varlimsup_{t\rightarrow\infty}\frac{t^2\, m_u'(e^t)}{(m_u(e^t))^2}\medskip\\
\displaystyle \ \ \le \varlimsup_{t\rightarrow\infty}\frac{t^2\, (m_u(e^t))^{\frac 43}\, (\phi_f(t))^{\frac 13}}{(m_u(e^t))^2}\le \varlimsup_{t\rightarrow\infty}\frac{t^2\, (\psi(t))^{\frac{1}{3\kappa(t)}}}{(1-\delta(t))^{\frac 23}\bigl(e^{-1}\psi(t)\bigr)^{\frac{2}{3\kappa(t-2\tilde s(t))}})}\medskip\\
\displaystyle \ \ \le
\varlimsup_{t\rightarrow\infty}\frac{e^{\frac 23}\,t^2}{(\psi(t))^{\frac{2}{3\kappa(t-2\tilde s(t))}-\frac{1}{3\kappa(t)}}}\le \varlimsup_{t\rightarrow\infty}\frac{e^{\frac 23}\,t^2}{(\ln\psi(t))^2}=0.
\end{array}
\]
This shows that $h:=e^f$ satisfies condition \eqref{cII}, i.e. $h\in\mathscr C$.

Further, under the hypotheses of the proposition, we show that $\sin f$ and $\cos f$ are of class $\mathscr C$. In fact,
by the definition of the trigonometric functions, for all sufficiently large $r$,
\[
\frac 13 M_{e^{if}}(r)\le\frac{1}{2} M_{e^{if}}(r)-\frac 12 M_{e^{-if}}(r)\le \max\bigl(M_{\sin f}(r),M_{\cos f}(r)\bigr)\le M_{e^{if}}(r).
\]
Since, as we have proved, $e^{if}\in\mathscr C$, the above inequality and Proposition \ref{propo1.11}\,(1) imply that $\sin f$ and $\cos f$ are of class $\mathscr C$ as well.

To complete the proof of the proposition, it remains to show that if $f\in\mathscr C$ is a univariate entire function satisfying condition \eqref{eq1.11}, then its derivative and every antiderivative are of class $\mathscr C$ and satisfy \eqref{eq1.11}.

Note that according to the Cauchy estimates for the derivative of a holomorphic function, for $0<s<1$ and all $r>0$,
\begin{equation}\label{eq7.59}
M_{f'}(sr)\le \frac{1}{1-s}M_f(r).
\end{equation}
On the other hand, by the mean-value theorem
\begin{equation}\label{eq7.60}
M_f(r)-|f(0)|\le r M_{f'}(r).
\end{equation}

First, assume that $f\in\mathscr C$ satisfies conditions \eqref{cI} and \eqref{eq1.11}. Then for some $C>0$ and all sufficiently large $t$,
\begin{equation}\label{eq7.61}
\frac{\phi_f(t+1)-\phi_f(t)}{\phi_f(t)-\phi_f(t-1)}<C.
\end{equation}
Applying \eqref{eq7.59} with $s=e^{-\frac 12}$, \eqref{eq7.60} and then \eqref{eq7.61} and convexity of $\phi_f$ we obtain, for all sufficiently large $t$,\smallskip
\begin{equation}\label{eq7.62}
\begin{array}{l}
\displaystyle
\frac{\phi_{f'}(t+1)-\phi_{f'}(t)}{\phi_{f'}(t)-\phi_{f'}(t-1)}\le \frac{\phi_f(t+\frac 32)-\phi_f(t)+t+ c_1}{\phi_f(t)-\phi_f(t-\frac 12)-t+c_2} \le\frac{(C+1)(\phi_f(t+\frac 12)-\phi_f(t-\frac 12))+t+ c_1}{\frac{1}{2}\bigl(\phi_f(t-\frac 12)-\phi_f(t-\frac 32)\bigr)-t+c_2}\medskip\\
\displaystyle\qquad\qquad\qquad\qquad\ \
\le\frac{C(C+1)(\phi_f(t-\frac 12)-\phi_f(t-\frac 32))+t+ c_1}{
\frac{1}{2}\bigl(\phi_f(t-\frac 12)-\phi_f(t-\frac 32)\bigr)-t+c_2}
\end{array}
\end{equation}
for some absolute constants $c_1,c_2\in\mathbb R$.

Further, due to convexity of $\phi_f$, for all $t>0$,
\[
\phi_f(t-\mbox{$\frac 12$})-\phi_f(t-\mbox{$\frac 32$})\ge\frac{\phi_f(t-\frac 32)-\phi_f(0)}{t-\frac 32}.
\]
This and condition \eqref{eq1.11} imply\smallskip
\begin{equation}\label{eq7.63}
\begin{array}{l}
\displaystyle
0\le\varliminf_{t\rightarrow\infty}\frac{t}{\phi_f(t-\frac 12)-\phi_f(t-\frac 32)}\le
\varlimsup_{t\rightarrow\infty}\frac{t}{\phi_f(t-\frac 12)-\phi_f(t-\frac 32)}\le \lim_{t\rightarrow\infty}\frac{t (t-\frac 32)}{\phi_f(t-\frac 32)-\phi_f(0)}\medskip\\
\displaystyle \ \ =\lim_{t\rightarrow\infty}\frac{(t-\frac 32)^2}{\phi_f(t-\frac 32)}=0.
\end{array}
\end{equation}
Therefore from \eqref{eq7.62}, \eqref{eq7.63} we deduce
\[
\varlimsup_{t\rightarrow\infty}\frac{\phi_{f'}(t+1)-\phi_{f'}(t)}{\phi_{f'}(t)-\phi_{f'}(t-1)}
\le\varlimsup_{t\rightarrow\infty}\frac{C(C+1)(\phi_f(t-\frac 12)-\phi_f(t-\frac 32))+t+ c_1}{
\frac{1}{2}\bigl(\phi_f(t-\frac 12)-\phi_f(t-\frac 32)\bigr)-t+c_2}\le 2C(C+1).
\]
This is condition \eqref{cI} for $f'$, i.e. $f'\in\mathscr C$.

Let us show that $f'$ satisfies condition \eqref{eq1.11}. Indeed, using \eqref{eq7.60} and condition \eqref{eq1.11} for $f$ we get, for all sufficiently large $t$,
\[
\varliminf_{t\rightarrow\infty}\frac{\phi_{f'}(t)}{t^2}\ge \varliminf_{t\rightarrow\infty} \frac{\phi_f(t)-t+c}{t^2}=\lim_{t\rightarrow\infty}\frac{\phi_f(t)}{t^2}=\infty,
\]
as required.

Now, let us prove that if $f'\in\mathscr C$ and satisfies \eqref{cI}, \eqref{eq1.11}, then its antiderivative $f\in\mathscr C$ and satisfy these conditions as well.

As before, we apply inequalities \eqref{eq7.59}, \eqref{eq7.60} and instead of \eqref{eq7.61} we use the condition
\begin{equation}\label{eq7.64}
\frac{\phi_{f'}(t+1)-\phi_{f'}(t)}{\phi_{f'}(t)-\phi_{f'}(t-1)}<C.
\end{equation}
Similarly to \eqref{eq7.62} we derive, for all sufficiently large $t$, that
\begin{equation}\label{eq7.65}
\begin{array}{l}
\displaystyle
\frac{\phi_{f}(t+1)-\phi_{f}(t)}{\phi_{f}(t)-\phi_{f}(t-1)}\le \frac{\phi_{f'}(t+1)-\phi_{f'}(t-\frac 12)+t+c_3}{\phi_{f'}(t-\frac 12)-\phi_{f'}(t-1)-t+c_4}\medskip\\
\displaystyle \qquad\qquad\qquad\qquad\le \frac{C(C+1)(\phi_{f'}(t-1)-\phi_{f'}(t-2))+t+ c_3}{
\frac{1}{2}\bigl(\phi_{f'}(t-1)-\phi_{f'}(t-2)\bigr)-t+c_4}
\end{array}
\end{equation}
for some absolute constants $c_3,c_4\in\mathbb R$. 

From here, as above (cf. \eqref{eq7.63}), we obtain that
\[
\varlimsup_{t\rightarrow\infty}\frac{\phi_{f}(t+1)-\phi_{f}(t)}{\phi_{f}(t)-\phi_{f}(t-1)}
\le 2C(C+1),
\]
i.e. $f\in\mathscr C$.

Also, since $f'$ satisfies \eqref{eq1.11}, using \eqref{eq7.59} we get
\[
\varliminf_{t\rightarrow\infty}\frac{\phi_{f}(t)}{t^2}\ge \varliminf_{t\rightarrow\infty} \frac{\phi_{f'}(t-\frac 12)-1}{t^2}=\lim_{t\rightarrow\infty}\frac{\phi_{f'}(t)}{t^2}=\infty.
\]
Hence, $f$ satisfies \eqref{eq1.11} as well and so the required statements of the proposition are proved for functions of finite order.

Next, we consider the case of $f\in\mathscr C$ satisfying condition \eqref{cII}. We apply inequalities \eqref{eq7.59}, \eqref{eq7.60} for $s=e^{-2\tilde s(t)}$ with $\tilde s(t)$ as in \eqref{eq7.56}, \eqref{eq7.57}, i.e. $\tilde s(t):=\min(s(t),1)$, where $s: [t_0,\infty)\rightarrow\mathbb R$, for some $t_0>0$, is a continuous function such that
\begin{equation}\label{eq7.66}
\frac{\psi(t)}{\psi(t-2s(t))}=e.
\end{equation}
Then we have
\begin{equation}\label{eq7.67}
\frac{1}{\tilde s(t)}\le \bigl(\phi_f(t)\bigr)^{\frac{\kappa(t)\varepsilon(t)}{t}},\qquad t\ge t_0,
\end{equation}
where $\varepsilon\in C\bigl([t_0,\infty)\bigr)$ is a positive continuous function tending to zero at $\infty$ and $\kappa$ is determined by \eqref{eq7.55}.

As before, each $t\ge t_0$ can be written as
$t=v_t-2\tilde s(v_t)$ for some $v_t> t$. Hence, from \eqref{eq7.59}, \eqref{eq7.66}, \eqref{eq7.65} and \eqref{eq7.55} we obtain, for all sufficiently large $t$,
\begin{equation}\label{eq7.68}
\begin{array}{l}
\displaystyle
\phi_{f'}(t)=\phi_{f'}(v_t-2\tilde s(v_t))\le \phi_f (v_t)+\ln\left(\frac{1}{\tilde s(v_t)}\right) +c\le 2\phi_f(v_t)=2(\psi(v_t))^{\frac{1}{\kappa(v_t)}}\medskip\\
\displaystyle \qquad\ \  \le 2\bigl(e\psi(v_t-2\tilde s(v_t))\bigr)^{\frac{1}{\kappa(v_t)}}=2e^{\frac{1}{\kappa(v_t)}}\bigl(\phi_f(t)\bigr)^{\frac{\kappa(t)}{\kappa(v_t)}}\le 6(\phi_f(t))^{\tilde\kappa(t)};
\end{array}
\end{equation}
here $c$ is an absolute constant and $\kappa(t)$ tends to one as $t\rightarrow\infty$.

Further, from \eqref{eq7.60} we deduce that, for all sufficiently large $t$,
\begin{equation}\label{eq7.69}
\phi_f(t)\le t+1+ \phi_{f'}(t).
\end{equation}
Equations \eqref{eq7.68}, \eqref{eq7.69} imply that
\[
\lim_{t\rightarrow\infty}\frac{\ln\phi_f(t)}{\ln\phi_{f'}(t)}=1.
\]
Therefore due to condition (II), see \eqref{cII}, function $f\in\mathscr C$ if and only if $f'\in\mathscr C$.

This completes the proof of the proposition. 
\end{proof}

\section{Proofs of Theorem \ref{theo1.11}, Proposition \ref{prop1.15} and Corollary \ref{cor1.15}}
\begin{proof}[{\bf 8.1. Proof of Theorem \ref{theo1.11}}]
First, we show that the the radius of convergence $r_f$ of the Taylor expansion of $f$ at $0$ is $\infty$, i.e. that $f$ is an entire function. By definition,
\[
\ln r_f=\varliminf_{j\rightarrow\infty}\frac{-\ln |c_j|}{j}=\varliminf_{j\rightarrow\infty}\frac{1}{j}\int_0^j h^{-1}(s)\,ds\ge \varliminf_{j\rightarrow\infty}\frac{1}{j}\int_{\frac j2}^j h^{-1}(s)\,ds\ge \varliminf_{j\rightarrow\infty}\frac{h^{-1}(\frac j2)}{2}=\infty,
\]
as required.

Next, we prove that $f$ is of finite order. Observe that the second condition  for $h$, see \eqref{eq1.12}, implies that for some $c>0$ and all  sufficiently large $t>0$, $h(t)\le e^{ct}$.
Passing in this inequality to inverse functions we obtain for all 
  $t\ge t_0$, for some $t_0\in\mathbb R_+$,
\[
h^{-1}(t)\ge \frac{\ln t}{c}.
\]
From here, by the definition of the order of $f$, we get
\[
\rho_f:=\varlimsup_{j\rightarrow\infty}\frac{j\ln j}{-\ln |c_j|}=\varlimsup_{j\rightarrow\infty}\frac{j\ln j}{\int_0^j h^{-1}(s)\,ds}\le \varlimsup_{j\rightarrow\infty}\frac{j\ln j}{\int_{t_0}^j h^{-1}(s)\,ds}\le \varlimsup_{j\rightarrow\infty}\frac{c j\ln j}{\int_{t_0}^j \ln s\,ds}=c,
\]
as required.

Since $\rho_f\le c$, for all sufficiently large $t>0$ we have, see, e.g., \cite[Ch.\,I.2]{L},
 \begin{equation}\label{equ8.45}
 \nu_f(t)\le\phi_f(t)\le 2c t+\nu_f(t),
 \end{equation}
 where
 \[
 \nu_f(t):=\sup_{j\in\mathbb N}\,(\ln |c_j|+jt),\qquad t\in\mathbb R_+.
 \]
 Hence,
 \begin{equation}\label{eq8.42}
 \begin{array}{lr}
 \displaystyle
 \phi_f(t+1)-\phi_f(t)\le \nu_f(t+1)-\nu_f(t)+2c(t+1);\\ \\
\displaystyle
 \phi_f(t)-\phi_f(t-1)\ge \nu_f(t)-\nu_f(t-1)-2c(t-1).
 \end{array}
 \end{equation}
 Let us consider the function
 \[
 g(x,t)=-\int_0^x h^{-1}(s)\, ds+xt,\quad (x, t)\in \mathbb R_+\times\mathbb R_+.
 \]
 One easily checks that for a fixed $t\ge h^{-1}(0)$ the function $g(\cdot, t)$ attains it maximal value at $x=h(t)$.
Using the substitution $s\mapsto h(s)$ and then the integration by parts we obtain
 \[
 \begin{array}{l}
 \displaystyle
 g(h(t),t)=
 -\int_0^{h(t)} h^{-1}(s)\, ds+h(t)\,t=-\int_{h^{-1}(0)}^{h^{-1}(h(t))} s h'(s)\, ds+h(t)\,t\medskip\\
 \displaystyle \qquad \qquad =
 h(t) (t-h^{-1}(h(t)))+\int_{h^{-1}(0)}^{h^{-1}(h(t))}h(s)\, ds=\int_{h^{-1}(0)}^{t}h(s)\,ds.
 \end{array}
 \]
  In particular, for $t\ge h^{-1}(0)$,
 \[
g(\lfloor h(t)\rfloor,t) =
-\int_0^{\lfloor h(t)\rfloor} h^{-1}(s)\, ds+\lfloor h(t)\rfloor\,t\le \nu_f(t)\le g(h(t),t)=\int_{h^{-1}(0)}^{t}h(s)\,ds.
 \]
 Also, by definition, for such $t$,
 \[
 \begin{array}{l}
 \displaystyle
0\le g(h(t),t)-g(\lfloor h(t)\rfloor,t)=\{h(t)\}\,t-\int_{\lfloor h(t)\rfloor}^{h(t)} h^{-1}(s)\,ds=\int_{\lfloor h(t)\rfloor}^{h(t)} \bigl(h^{-1}(h(t))-h^{-1}(s)\bigr)\,ds\medskip\\
 \displaystyle \ \ \le \omega_{h(t)}(1;h^{-1}) \le t.\end{array}
 \]
This yields (for all $t\ge h^{-1}(0)$)
\begin{equation}\label{eq8.43}
\int_{h^{-1}(0)}^{t}h(s)\,ds- t\le \int_{h^{-1}(0)}^{t}h(s)\,ds-\omega_{h(t)}(1;h^{-1})\le
\nu_f(t)\le \int_{h^{-1}(0)}^{t}h(s)\,ds.
\end{equation}
Using these estimates in \eqref{eq8.42} we get, for all sufficiently large $t>0$,
\begin{equation}\label{equ8.48}
\phi_f(t+1)-\phi_f(t)\le \int_t^{t+1}h(s)\, ds+t+2c (t+1)\le h(t+1)+(2c+1)(t+1)
\end{equation}
and
\begin{equation}\label{equ8.49}
\phi_f(t)-\phi_f(t-1)\ge \int_{t-1}^t h(s)\, ds- t-2c(t-1)\ge h(t-1)- (2c+1)t.
\end{equation} 
Invoking properties of $h$, we derive from the last two inequalities that
\[
\varlimsup_{t\rightarrow\infty}\frac{\phi_f(t+1)-\phi_f(t)}{\phi_f(t)-\phi_f(t-1)}\le\varlimsup_{t\rightarrow\infty}\frac{h(t+1)}{h(t-1)}<\infty.
\]
Thus, $f\in\mathscr C$. 

Also, due to \eqref{eq1.12},
\[
\lim_{t\rightarrow\infty}\frac{\phi_f(t)}{t^2}\ge\lim_{t\rightarrow\infty}\frac{\nu_f(t)}{t^2}\ge \lim_{t\rightarrow\infty}\frac{\int_{h^{-1}(0)}^{t}h(s)\,ds- t}{t^2}\ge \lim_{t\rightarrow\infty}\frac{\int_{\frac t2}^{t}h(s)\,ds}{t^2}\ge  \lim_{t\rightarrow\infty}\frac{h\bigl(\frac t2\bigr)}{2t}=\infty.
\]
Hence, $f$ satisfies condition \eqref{eq1.11}.

Finally,
\[
\rho_f=\varlimsup_{t\rightarrow\infty}\frac{\ln\phi_f(t)}{t}=\varlimsup_{t\rightarrow\infty}\frac{\ln\nu_f(t)}{t}.
\]
Therefore from \eqref{eq8.43} we obtain
\[
\rho_f= \varlimsup_{t\rightarrow\infty}\frac{\ln h(t)}{t}.
\]
The proof of the theorem is complete.
\end{proof}
\begin{proof}[{\bf 8.2. Proof of Proposition \ref{prop1.15}}]
Clearly, it suffices to prove that under the hypotheses of the proposition, $f_{h_1}+cf_{h_2}\in\mathscr C$ for all $c\in\mathbb C\setminus 0$. 

Due to \cite[Ch.\,I.2,\,Eq.\,(1.10)]{L} (cf. \eqref{equ8.45} above) and \eqref{eq8.43}, for each
$C>\max(\rho_{f_{h_1}},\rho_{f_{h_2}})$ there exists $t_C>0$ such that for all $t\ge t_C$,
\begin{equation}\label{eq8.48}
\int_{h_i^{-1}(0)}^t h_i(s)\, ds-\omega_{h_i(t)}(1;h_i^{-1})
\le \phi_{f_{h_i}}(t)\le Ct+\int_{h_i^{-1}(0)}^t h_i(s)\, ds,\qquad i=1,2.
\end{equation}

Also, the assumption of the proposition implies that for some $q>\displaystyle\varlimsup_{t\rightarrow\infty}\mbox{$\frac{\omega_t(1;h_1^{-1})}{h_1^{-1}(t)}$}=:\omega_{h_1}$ and all sufficiently large $t>0$,
\begin{equation}\label{eq8.49}
h_2(t)< h_1(t)-q-\rho_{f_{h_1}}.
\end{equation}
In particular, we obtain that $\rho_{f_{h_2}}\le\rho_{f_{h_1}}$, see Theorem \ref{theo1.11}.

Inequality \eqref{eq8.49} shows that for each $\tilde q\in \left(\omega_{h_1},q\right)$  there exists $t_{\tilde q}>0$ such that for all $t\ge t_{\tilde q}$,
\[
\int_{h_2^{-1}(0)}^t h_2(s)\, ds< \int_{h_1^{-1}(0)}^t h_1(s)\, ds - (\tilde q+\rho_{f_{h_1}})t.
\]
From here and \eqref{eq8.48}  with $C:=\rho_{f_{h_1}}+\frac{\tilde q-\omega_{h_1}}{2}$ we obtain, for all sufficiently large $t>0$,
\[
\begin{array}{l}
\displaystyle
\phi_{f_{h_2}}(t)< \left(C-\tilde q-\rho_{f_{h_1}}+\frac{\omega_{h_1(t)}(1;h_1^{-1})}{t}\right)t +\phi_{f_{h_1}}(t)\medskip\\
\displaystyle \qquad \quad <
\left(-\frac{\tilde q+\omega_{h_1}}{2}+\omega_{h_1}+\frac{\tilde q-\omega_{h_1}}{4}\right)t +\phi_{f_{h_1}}(t)=\frac{\omega_{h_1}-\tilde q}{4}\, t +\phi_{f_{h_1}}\medskip\\
\displaystyle\qquad\quad<-\ln (1+|c|)+\phi_{f_{h_1}}(t).
\end{array}
\]
Thus for all sufficiently large $r>0$,
\[
M_{cf_{h_2}}(r)<\frac{|c|}{1+|c|}\, M_{f_{h_1}}(r).
\]
This implies that
\[
\frac{1}{1+|c|}\le\varliminf_{r\rightarrow\infty}\frac{M_{f_{h_1}+cf_{h_2}}(r)}{ M_{f_{h_1}}(r)}\le\varlimsup_{r\rightarrow\infty}\frac{M_{f_{h_1}+cf_{h_2}}(r)}{ M_{f_{h_1}}(r)}\le 2.
\]
Therefore $f_{h_1}+cf_{h_2}\in\mathscr C$ (cf. property (1) in section 1.4).

The proof of the proposition is complete.
\end{proof}
\begin{proof}[{\bf 8.3. Proof of Corollary \ref{cor1.15}}]
We use inequalities \eqref{equ8.48}, \eqref{equ8.49} for functions
$h_1,\dots, h_l$. Then we have, for a fixed $c>\max\{\rho_{f_{h_1}},\dots,\rho_{f_{h_l}}\}$ and all sufficiently large $t>0$,
\[
h_j(t-1)-(2c+1)t\le\phi_{f_{h_j}}(t)-\phi_{f_{h_j}}(t-1)\le h_j(t+1)+(2c+1)(t+1).
\]
Together with conditions \eqref{eq1.12} for functions $h_j$, $1\le j\le l$, this implies 
\[
\begin{array}{l}
\displaystyle
\varlimsup_{r\rightarrow\infty}\frac{m_{f_j}(r)-m_{f_j}\bigl(\frac re\bigr)}{\sqrt{m_{f_{j+1}}(r)-m_{f_{j+1}}\bigl(\frac{r}{e}\bigr)}}\ge \varlimsup_{t\rightarrow\infty}\frac{h_j(t-1)-(2c+1)t}{\sqrt{h_{j+1}(t+1)+(2c+1)(t+1)}}\ge
A^{-\frac 32}\varlimsup_{t\rightarrow\infty}\frac{h_j(t)}{\sqrt{h_{j+1}(t)}};\medskip\\
\displaystyle
\varlimsup_{r\rightarrow\infty}\frac{m_{f_j}(r)-m_{f_j}\bigl(\frac re\bigr)}{\sqrt{m_{f_{j+1}}(r)-m_{f_{j+1}}\bigl(\frac{r}{e}\bigr)}}\le \varlimsup_{t\rightarrow\infty}\frac{h_j(t+1)+(2c+1)(t+1)}{\sqrt{h_{j+1}(t-1)-(2c+1)t}}\le
A^{\frac 32}\varlimsup_{t\rightarrow\infty}\frac{h_j(t)}{\sqrt{h_{j+1}(t)}},
\end{array}
\]
where 
\[
A:=\max_{1\le j\le l}\left\{\varlimsup_{t\rightarrow\infty}\frac{h_j(t+1)}{h_j(t)}\right\}.
\]
Hence, condition \eqref{eq1.4} of Theorem \ref{theo1.8} acquires the form
\[
0=\lim_{r\rightarrow\infty}\frac{m_{f_j}(r)-m_{f_j}\bigl(\frac re\bigr)}{\sqrt{m_{f_{j+1}}(r)-m_{f_{j+1}}\bigl(\frac{r}{e}\bigr)}}=\lim_{t\rightarrow\infty}\frac{h_j(t)}{\sqrt{h_{j+1}(t)}}\quad {\rm for\ all}\quad 1\le j\le l-1.
\]

Further, let $u_j:={\rm Re} f_{h_j}$, $l+1\le j\le m$.
Then $\ln m_{e^{f_{h_j}}}(r)=m_{u_j}(r)$ for all such $ j$. Using the Borel-Carath\'eodory theorem (cf. \eqref{eq7.35}) we obtain, for $0<s<1$ and all $r>0$,
\begin{equation}\label{eq8.77}
m_{h_j}(sr)\le m_{u_j}(r)-\ln(1-s)+c_j
\end{equation}
for some constant $c_j:=c(h_j)$.

On the other hand,
\begin{equation}\label{eq8.78}
m_{u_j}(r)\le m_{h_j}(r).
\end{equation}
Applying \eqref{eq8.77} with $s=1-e^{-t}$, \eqref{eq8.78} together with \eqref{equ8.45}, \eqref{eq8.43} for functions $h_j$, $l+1\le j\le m$, we obtain, for all sufficiently large $r:=e^t$,\smallskip
\begin{equation}\label{eq8.79}
\begin{array}{l}
\displaystyle \frac{m_{u_j}(r)}{m_{u_{j+1}}(\frac re)}\le \frac{\phi_{h_j}(t)}{\phi_{h_{j+1}}(t-1+\ln(1-e^{-t}))-t+c_1}\le \frac{c_2t+\int_{h_j^{-1}(0)}^t h_j(s)\,ds}{\int_{h_{j+1}^{-1}(0)}^{t-1-2e^{-t}} h_{j+1}(s)\,ds -2t+c_3}\medskip\\
\displaystyle  \frac{m_{u_j}(r)}{m_{u_{j+1}}(\frac re)}\ge\frac{\phi_{h_j}(t+\ln(1-e^{-t}))-t+c_4}{\phi_{h_{j+1}}(t-1)}\ge\frac{\int_{h_{j}^{-1}(0)}^{t-2e^{-t}} h_{j}(s)\,ds-2t+c_5}{c_6t+\int_{h_{j+1}^{-1}(0)}^{t-1} h_{j+1}(s)\,ds}.
\end{array}
\end{equation}
for some constants $c_1,\dots, c_6$.

To proceed we require
\begin{Lm}\label{lem8.1}
We have
\[
 \lim_{t\rightarrow\infty}\frac{\int_{h_{j}^{-1}(0)}^{t-2e^{-t}} h_{j}(s)\,ds}{\int_{h_{j}^{-1}(0)}^{t} h_{j}(s)\,ds}=1\quad {\rm and}\quad \lim_{t\rightarrow\infty}\frac{\int_{h_{j+1}^{-1}(0)}^{t-1-2e^{-t}} h_{j+1}(s)\,ds}{\int_{h_{j+1}^{-1}(0)}^{t-1} h_{j+1}(s)\,ds}=1.
\]
\end{Lm}
\begin{proof}
By the definition of $h_j$, see \eqref{eq1.12}, for all sufficiently large $t$,
\[
ce^{-t}h_j(t)\le 2e^{-t}h_j(t-1)\le \int_{t-2e^{-t}}^t h_{j}(s)\,ds\le 2e^{-t}h_j(t)
\]
for some constant $c\in (0,1)$.

On the other hand,
\[
ch_j(t)\le h_j(t-1)\le\int_{h_{j}^{-1}(0)}^{t} h_{j}(s)\,ds\le t h_j(t).
\]
Comparing these inequalities, we obtain the first statement of the lemma. The second statement can be proved analogously.
\end{proof}

Using this lemma together with \eqref{eq8.79} and \eqref{eq1.12} we get, for $l+1\le j\le m-1$,
\[
\varlimsup_{r\rightarrow\infty}\frac{\ln m_{e^{h_j}}(r)}{\ln m_{e^{h_{j+1}}}(r)}=\varlimsup_{r\rightarrow\infty}\frac{m_{u_j}(r)}{m_{u_{j+1}}(\frac re)}=\varlimsup_{t\rightarrow\infty}\frac{\int_{M}^{t} h_{j}(s)\,ds}{\int_{M}^{t-1} h_{j+1}(s)\,ds},
\]
where $M:=\max_{l+1\le j\le m}\{h_j^{-1}(0)\}$.

This expression and Theorem \ref{theo1.8} give the required statement.

The proof of the corollary is complete.
\end{proof}

\end{document}